\documentclass[12pt]{amsart}
\usepackage{amssymb,latexsym,tikz}
\usepackage[top=40mm, bottom=25mm, left=35mm, right=35mm]{geometry}
\usetikzlibrary{arrows,decorations.markings}
\tikzstyle arrowstyle=[scale=1]
\tikzstyle directed=[postaction={decorate,
decoration={markings,mark=at position .65 with {\arrow[arrowstyle]{stealth}}}}]
\usepackage{verbatim}
\usepackage{enumitem}
\usepackage{mathtools}

\usetikzlibrary{decorations.markings}
\tikzset{mid vert/.style={/utils/exec=\tikzset{every node/.append style={outer sep=0.8ex}},
postaction=decorate,decoration={markings,
mark=at position 0.5 with {\draw[-] (0,#1) -- (0,-#1);}}},
mid vert/.default=0.75ex}

\usepackage{amscd}
\usepackage{xypic}
\usepackage[utf8]{inputenc}

\begin{document}

\title
[Cointeracting bialgebras of hypergraphs and chromatic polynomials]
{Eight times four bialgebras of hypergraphs, cointeractions, and chromatic
polynomials}

\author{Kurusch Ebrahimi-Fard}
\address{Department of Mathematics, NTNU\\
         Norway}
\email{kurusch.ebrahimi-fard@ntnu.no}

\author{Gunnar Fl{\o}ystad}
\address{Matematisk institutt,
  Universitetet i Bergen \\
         Postboks 7803,
         5020 Bergen \\
       Norway}
\email{gunnar@mi.uib.no}

\begin{abstract}
The bialgebra of hypergraphs, a generalization of
W. Schmitt's Hopf algebra of graphs \cite{Schmitt},
is shown to have a cointeracting
bialgebra, giving a double bialgebra in the sense of L.~Foissy,
who has recently proven \cite{Fo22} there is then a unique double bialgebra morphism to  the double bialgebra structure on the polynomial ring ${\mathbb Q}[x]$.

  We show that the associated polynomial is the
  hypergraph chromatic polynomial. Moreover, hypergraphs occur in
  quartets, i.e., there is a dual, a complement, and a dual complement hypergraph.
  These correspondences are involutions and give rise to three other
  double bialgebras, and three more chromatic polynomials. In all
  we give eight quartets of bialgebras, which includes
  recent bialgebras of hypergraphs introduced by M.~Aguiar and F.~Ardila
  \cite{AA},  and by L.~Foissy \cite{Fo-Hyp}.
\end{abstract}


\keywords{hypergraph, double bialgebra, chromatic polynomial, cointeracting
  bialgebras,
extraction/contraction}
\subjclass[2010]{Primary: 16T30, 05C15 ; Secondary: 16T10}
\date{\today}


\theoremstyle{plain}
\newtheorem{theorem}{Theorem}[section]
\newtheorem{corollary}[theorem]{Corollary}
\newtheorem*{main}{Main Theorem}
\newtheorem{lemma}[theorem]{Lemma}
\newtheorem{proposition}[theorem]{Proposition}
\newtheorem{conjecture}[theorem]{Conjecture}
\newtheorem{theoremp}{Theorem}

\theoremstyle{definition}
\newtheorem{definition}[theorem]{Definition}
\newtheorem{fact}[theorem]{Fact}
\newtheorem{obs}[theorem]{Observation}
\newtheorem{definisjon}[theorem]{Definisjon}
\newtheorem{problem}[theorem]{Problem}
\newtheorem{condition}[theorem]{Condition}

\theoremstyle{remark}
\newtheorem{notation}[theorem]{Notation}
\newtheorem{remark}[theorem]{Remark}
\newtheorem{example}[theorem]{Example}
\newtheorem{claim}{Claim}
\newtheorem{observation}[theorem]{Observation}
\newtheorem{question}[theorem]{Question}


\newcommand{\psp}[1]{{{\bf P}^{#1}}}
\newcommand{\psr}[1]{{\bf P}(#1)}
\newcommand{\op}{{\mathcal O}}
\newcommand{\opw}{\op_{\psr{W}}}

\newcommand{\ini}[1]{\text{in}(#1)}
\newcommand{\gin}[1]{\text{gin}(#1)}
\newcommand{\kr}{{\Bbbk}}
\newcommand{\pd}{\partial}
\newcommand{\vardel}{\partial}
\renewcommand{\tt}{{\bf t}}


\newcommand{\coh}{{{\text{{\rm coh}}}}}


\newcommand{\modv}[1]{{#1}\text{-{mod}}}
\newcommand{\modstab}[1]{{#1}-\underline{\text{mod}}}

\newcommand{\sut}{{}^{\tau}}
\newcommand{\sumit}{{}^{-\tau}}
\newcommand{\til}{\thicksim}

\newcommand{\totp}{\text{Tot}^{\prod}}
\newcommand{\dsum}{\bigoplus}
\newcommand{\dprod}{\prod}
\newcommand{\lsum}{\oplus}
\newcommand{\lprod}{\Pi}

\newcommand{\La}{{\Lambda}}

\newcommand{\sirstj}{\circledast}

\newcommand{\she}{\EuScript{S}\text{h}}
\newcommand{\cm}{\EuScript{CM}}
\newcommand{\cmd}{\EuScript{CM}^\dagger}
\newcommand{\cmri}{\EuScript{CM}^\circ}
\newcommand{\cler}{\EuScript{CL}}
\newcommand{\clerd}{\EuScript{CL}^\dagger}
\newcommand{\clerri}{\EuScript{CL}^\circ}
\newcommand{\gor}{\EuScript{G}}
\newcommand{\cF}{\mathcal{F}}
\newcommand{\cG}{\mathcal{G}}
\newcommand{\cM}{\mathcal{M}}
\newcommand{\cE}{\mathcal{E}}
\newcommand{\cI}{\mathcal{I}}
\newcommand{\cP}{\mathcal{P}}
\newcommand{\cK}{\mathcal{K}}
\newcommand{\cS}{\mathcal{S}}
\newcommand{\cC}{\mathcal{C}}
\newcommand{\cO}{\mathcal{O}}
\newcommand{\cJ}{\mathcal{J}}
\newcommand{\cU}{\mathcal{U}}
\newcommand{\cQ}{\mathcal{Q}}
\newcommand{\cX}{\mathcal{X}}
\newcommand{\cY}{\mathcal{Y}}
\newcommand{\cZ}{\mathcal{Z}}
\newcommand{\cV}{\mathcal{V}}

\newcommand{\mm}{\mathfrak{m}}

\newcommand{\dlim} {\varinjlim}
\newcommand{\ilim} {\varprojlim}

\newcommand{\CM}{\text{CM}}
\newcommand{\Mon}{\text{Mon}}


\newcommand{\Kom}{\text{Kom}}


\newcommand{\EH}{{\mathbf H}}
\newcommand{\res}{\text{res}}
\newcommand{\Hom}{\text{Hom}}
\newcommand{\inhom}{{\underline{\text{Hom}}}}
\newcommand{\Ext}{\text{Ext}}
\newcommand{\Tor}{\text{Tor}}
\newcommand{\ghom}{\mathcal{H}om}
\newcommand{\gext}{\mathcal{E}xt}
\newcommand{\im}{\text{im}\,}
\newcommand{\codim} {\text{codim}\,}
\newcommand{\resol}{\text{resol}\,}
\newcommand{\rank}{\text{rank}\,}
\newcommand{\lpd}{\text{lpd}\,}
\newcommand{\coker}{\text{coker}\,}
\newcommand{\supp}{\text{supp}\,}
\newcommand{\Ad}{A_\cdot}
\newcommand{\Bd}{B_\cdot}
\newcommand{\Fd}{F_\cdot}
\newcommand{\Gd}{G_\cdot}


\newcommand{\sus}{\subseteq}
\newcommand{\sups}{\supseteq}
\newcommand{\pil}{\rightarrow}
\newcommand{\vpil}{\leftarrow}
\newcommand{\rpil}{\leftarrow}
\newcommand{\lpil}{\longrightarrow}
\newcommand{\inpil}{\hookrightarrow}
\newcommand{\pils}{\twoheadrightarrow}
\newcommand{\projpil}{\dashrightarrow}
\newcommand{\dotpil}{\dashrightarrow}
\newcommand{\adj}[2]{\overset{#1}{\underset{#2}{\rightleftarrows}}}
\newcommand{\mto}[1]{\stackrel{#1}\longrightarrow}
\newcommand{\vmto}[1]{\stackrel{#1}\longleftarrow}
\newcommand{\mtoelm}[1]{\stackrel{#1}\mapsto}
\newcommand{\bihom}[2]{\overset{#1}{\underset{#2}{\rightleftarrows}}}
\newcommand{\eqv}{\Leftrightarrow}
\newcommand{\impl}{\Rightarrow}

\newcommand{\iso}{\cong}
\newcommand{\te}{\otimes}
\newcommand{\into}[1]{\hookrightarrow{#1}}
\newcommand{\ekv}{\Leftrightarrow}
\newcommand{\equi}{\simeq}
\newcommand{\isopil}{\overset{\cong}{\lpil}}
\newcommand{\equipil}{\overset{\equi}{\lpil}}
\newcommand{\ispil}{\isopil}
\newcommand{\vvi}{\langle}
\newcommand{\hvi}{\rangle}
\newcommand{\susneq}{\subsetneq}
\newcommand{\sgn}{\text{sign}}


\newcommand{\xd}{\check{x}}
\newcommand{\ortog}{\bot}
\newcommand{\tL}{\tilde{L}}
\newcommand{\tM}{\tilde{M}}
\newcommand{\tH}{\tilde{H}}
\newcommand{\tvH}{\widetilde{H}}
\newcommand{\tvh}{\widetilde{h}}
\newcommand{\tV}{\tilde{V}}
\newcommand{\tS}{\tilde{S}}
\newcommand{\tT}{\tilde{T}}
\newcommand{\tR}{\tilde{R}}
\newcommand{\tf}{\tilde{f}}
\newcommand{\ts}{\tilde{s}}
\newcommand{\tp}{\tilde{p}}
\newcommand{\tr}{\tilde{r}}
\newcommand{\tfst}{\tilde{f}_*}
\newcommand{\empt}{\emptyset}
\newcommand{\bfa}{{\mathbf a}}
\newcommand{\bfb}{{\mathbf b}}
\newcommand{\bfd}{{\mathbf d}}
\newcommand{\bfl}{{\mathbf \ell}}
\newcommand{\bfx}{{\mathbf x}}
\newcommand{\bfm}{{\mathbf m}}
\newcommand{\bfv}{{\mathbf v}}
\newcommand{\bft}{{\mathbf t}}
\newcommand{\bbfa}{{\mathbf a}^\prime}
\newcommand{\la}{\lambda}
\newcommand{\bfen}{{\mathbf 1}}
\newcommand{\bfe}{{\mathbf 1}}
\newcommand{\ep}{\epsilon}
\newcommand{\en}{r}
\newcommand{\tu}{s}
\newcommand{\Sym}{\text{Sym}}

\newcommand{\ome}{\omega_E}

\newcommand{\bevis}{{\bf Proof. }}
\newcommand{\demofin}{\qed \vskip 3.5mm}
\newcommand{\nyp}[1]{\noindent {\bf (#1)}}
\newcommand{\demo}{{\it Proof. }}
\newcommand{\demodone}{\demofin}
\newcommand{\parg}{{\vskip 2mm \addtocounter{theorem}{1}  
                   \noindent {\bf \thetheorem .} \hskip 1.5mm }}

\newcommand{\lcm}{{\text{lcm}}}


\newcommand{\dl}{\Delta}
\newcommand{\cdel}{{C\Delta}}
\newcommand{\cdelp}{{C\Delta^{\prime}}}
\newcommand{\dlst}{\Delta^*}
\newcommand{\Sdl}{{\mathcal S}_{\dl}}
\newcommand{\lk}{\text{lk}}
\newcommand{\lkd}{\lk_\Delta}
\newcommand{\lkp}[2]{\lk_{#1} {#2}}
\newcommand{\del}{\Delta}
\newcommand{\delr}{\Delta_{-R}}
\newcommand{\dd}{{\dim \del}}
\newcommand{\Del}{\Delta}

\renewcommand{\aa}{{\bf a}}
\newcommand{\bb}{{\bf b}}
\newcommand{\cc}{{\bf c}}
\newcommand{\xx}{{\bf x}}
\newcommand{\yy}{{\bf y}}
\newcommand{\zz}{{\bf z}}
\newcommand{\mv}{{\xx^{\aa_v}}}
\newcommand{\mF}{{\xx^{\aa_F}}}

\newcommand{\Symm}{\text{Sym}}
\newcommand{\pnm}{{\bf P}^{n-1}}
\newcommand{\opnm}{{\go_{\pnm}}}
\newcommand{\ompnm}{\omega_{\pnm}}

\newcommand{\pn}{{\bf P}^n}
\newcommand{\hele}{{\mathbb Z}}
\newcommand{\nat}{{\mathbb N}}
\newcommand{\rasj}{{\mathbb Q}}
\newcommand{\bfone}{{\mathbf 1}}

\newcommand{\dt}{\bullet}
\newcommand{\disk}{\scriptscriptstyle{\bullet}}

\newcommand{\cxF}{F_\dt}
\newcommand{\pol}{f}

\newcommand{\Rn}{{\mathbb R}^n}
\newcommand{\An}{{\mathbb A}^n}
\newcommand{\frg}{\mathfrak{g}}
\newcommand{\PW}{{\mathbb P}(W)}

\newcommand{\pos}{{\mathcal Pos}}
\newcommand{\g}{{\gamma}}

\newcommand{\Vaa}{V_0}
\newcommand{\Bp}{B^\prime}
\newcommand{\Bpp}{B^{\prime \prime}}
\newcommand{\bbp}{\mathbf{b}^\prime}
\newcommand{\bbpp}{\mathbf{b}^{\prime \prime}}
\newcommand{\bp}{{b}^\prime}
\newcommand{\bpp}{{b}^{\prime \prime}}

\newcommand{\oLa}{\overline{\Lambda}}
\newcommand{\ov}[1]{\overline{#1}}
\newcommand{\ovv}[1]{\overline{\overline{#1}}}
\newcommand{\tm}{\tilde{m}}
\newcommand{\po}{\bullet}

\newcommand{\surj}[1]{\overset{#1}{\twoheadrightarrow}}
\newcommand{\Supp}{\text{Supp}}

\def\CC{{\mathbb C}}
\def\GG{{\mathbb G}}
\def\ZZ{{\mathbb Z}}
\def\NN{{\mathbb N}}
\def\RR{{\mathbb R}}
\def\OO{{\mathbb O}}
\def\QQ{{\mathbb Q}}
\def\VV{{\mathbb V}}
\def\PP{{\mathbb P}}
\def\EE{{\mathbb E}}
\def\FF{{\mathbb F}}
\def\AA{{\mathbb A}}

\renewcommand{\SS}{{\mathcal S}}
\newcommand{\pr}{\preceq}
\newcommand{\su}{\succeq}
\newcommand{\bX}{{\mathbf X}}
\newcommand{\bY}{{\mathbf Y}}
\newcommand{\bA}{{\mathbf A}}
\newcommand{\bB}{{\mathbf B}}
\newcommand{\bH}{{\mathbf H}}
\renewcommand{\op}{{\text op}}
\newcommand{\Spol}{S}

\newcommand{\ben}{{\bf 1}}
\newcommand{\een}{{\bf 1}}
\newcommand{\id}{{\bf I}}
\newcommand{\ppil}{\twoheadrightarrow}
\newcommand{\eps}{\epsilon}
\newcommand{\vep}{\varepsilon}
\newcommand{\vareps}{\varepsilon}
\newcommand{\kk}{\Bbbk}
\newcommand{\Hy}{H}
\newcommand{\He}{H^\emptyset}
\newcommand{\Hyt}{H^\prime}
\newcommand{\Hytl}[1]{H^{\prime #1}}
\newcommand{\co}{\text{cor}}
\newcommand{\llin}{\raisebox{1pt}{\scalebox{1}[0.6]{$\mid$}}}
\newcommand{\promap}{\mathrlap{{\hskip 2.8mm}{\llin}}{\lpil}}

\maketitle


\newcommand{\cherry}{
\begin{tikzpicture}[scale=.5, vertices/.style={draw, fill=black, circle, inner sep=1.5pt}]
\draw [help lines, white] (-1,0) grid (1,1);
\node [vertices] (a) at (-0.5,0) {};
\node [vertices] (b) at (0,0.7) {};
\node [vertices] (c) at (0.5,0) {};

\foreach \to/\from in {a/b, b/c}
\draw (\to)--(\from);
\end{tikzpicture}
}

\newcommand{\cher}{
\begin{tikzpicture}[scale=.5, vertices/.style={draw, fill=black, circle, inner sep=1.5pt}]
\draw [help lines, white] (-1,0) grid (1,1);
\coordinate (a) at (-0.5,0) {};
\node [vertices] (b) at (0,0.7) {};
\node [vertices] (c) at (0.5,0) {};

\foreach \to/\from in {a/b, b/c}
\draw (\to)--(\from);
\end{tikzpicture}
}

\newcommand{\trev}{
\begin{tikzpicture}[scale=.5, vertices/.style={draw, fill=black, circle, inner sep=1.5pt}]
\draw [help lines, white] (-1,0) grid (1,1);
\node [vertices] (a) at (-0.5,0) {};
\node [vertices] (b) at (0,0.7) {};
\node [vertices] (c) at (0.5,0) {};

\end{tikzpicture}
}

\newcommand{\trekant}{
  \begin{tikzpicture}[scale=.5, vertices/.style={draw, fill, circle, inner sep=1.5pt}]

\draw [fill=blue, opacity=0.5] (-0.5,0)--(0,0.7)--(0.5,0)--(-0.5,0);

\node [vertices] (a) at (-0.5,0) {};
\node [vertices] (b) at (0,0.7) {};
\node [vertices] (c) at (0.5,0) {};

\foreach \to/\from in {a/b, b/c, c/a}
\draw (\to)--(\from);
\end{tikzpicture}
}

\newcommand{\trekantKant}{
  \begin{tikzpicture}[scale=.5, vertices/.style={draw, fill, circle, inner sep=1.5pt}]

\draw [fill=blue, opacity=0.5] (-0.5,0)--(0,0.7)--(0.5,0)--(-0.5,0);

\node [vertices] (a) at (-0.5,0) {};
\node [vertices] (b) at (0,0.7) {};
\node [vertices] (c) at (0.5,0) {};

\draw (a) to[out=90,in=200,looseness=1] (b);

\foreach \to/\from in {a/b, b/c, c/a}
\draw (\to)--(\from);
\end{tikzpicture}
}

\newcommand{\lokke}{
  \begin{tikzpicture}[scale=.5, vertices/.style={draw, fill, circle, inner sep=1.5pt}]

\node [vertices] (a) at (0,0) {};

\draw (a) to[out=120,in=60,looseness=8] (a);

\end{tikzpicture}
}

\newcommand{\linr}{
\begin{tikzpicture}[scale=.5, vertices/.style={draw, fill=black, circle, inner sep=1.5pt}]
\node [vertices] (a) at (-0.5,0) {};
\node [vertices] (b) at (0,0.7) {};

\foreach \to/\from in {a/b}
\draw (\to)--(\from);
\end{tikzpicture}
}

\newcommand{\linl}{
\begin{tikzpicture}[scale=.5, vertices/.style={draw, fill=black, circle, inner sep=1.5pt}]
\node [vertices] (c) at (0.5,0) {};
\node [vertices] (b) at (0,0.7) {};

\foreach \to/\from in {c/b}
\draw (\to)--(\from);
\end{tikzpicture}
}

\newcommand{\linvert}{
\begin{tikzpicture}[scale=.5, vertices/.style={draw, fill=black, circle, inner sep=1.5pt}]
\node [vertices] (a) at (0,0) {};
\node [vertices] (b) at (0,0.7) {};

\foreach \to/\from in {a/b}
\draw (\to)--(\from);
\end{tikzpicture}
}

\newcommand{\tobull}{
\begin{tikzpicture}[scale=.5, vertices/.style={draw, fill=black, circle, inner sep=1.5pt}]
\node [vertices] (a) at (0,0) {};
\node [vertices] (b) at (0,0.7) {};

\end{tikzpicture}
}

\newcommand{\linP}{
\begin{tikzpicture}[scale=.5, vertices/.style={draw, fill=black, circle, inner sep=1.5pt}]
\node [vertices] (a) at (0,0) {};
\node [vertices] (b) at (0,0.7) {};
\node [vertices] (c) at (0.7,0) {};

\foreach \to/\from in {a/b}
\draw (\to)--(\from);
\end{tikzpicture}
}

\newcommand{\pointl}{
\begin{tikzpicture}[scale=.5, vertices/.style={draw, fill=black, circle, inner sep=1.5pt}]
\node (b) at (0,1) {};
\node [vertices] (c) at (0,0) {};

\foreach \to/\from in {c/b}
\draw (\to)--(\from);
\end{tikzpicture}
}

\newcommand{\edge}{
\begin{tikzpicture}[scale=.5, vertices/.style={draw, fill=black, circle, inner sep=1pt}]
\coordinate (a) at (-1,0) {};
\coordinate (b) at (0,0) {};
\coordinate (c) at (0,0.7) {};
\draw (b)--(c);
\end{tikzpicture}
}

\newcommand{\pointll}{
\begin{tikzpicture}[scale=.5, vertices/.style={draw, fill=black, circle, inner sep=1.5pt}]
\coordinate (a) at (-0.3,0.7) {};
\coordinate (b) at (0.3,0.7) {};
\node [vertices] (c) at (0,0) {};

\foreach \to/\from in {c/b, c/a}
\draw (\to)--(\from);
\end{tikzpicture}
}

\newcommand{\pointlh}{
\begin{tikzpicture}[scale=.5, vertices/.style={draw, fill=black, circle, inner sep=1.5pt}]
\node [vertices] (c) at (0.5,0) {};

\foreach \to/\from in {c/b}
\draw (\to)--(\from);
\end{tikzpicture}
}

\newcommand{\pointrh}{
\begin{tikzpicture}[scale=.5, vertices/.style={draw, fill=black, circle, inner sep=1.5pt}]
\node [vertices] (a) at (-0.5,0) {};

\foreach \to/\from in {a/b}
\draw (\to)--(\from);
\end{tikzpicture}
}

\newcommand{\point}{
\begin{tikzpicture}[scale=.5, vertices/.style={draw, fill=black, circle, inner sep=1.5pt}]
\node [vertices] (a) at (-0,0) {};

\end{tikzpicture}
}

\newcommand{\dpoint}{
\begin{tikzpicture}[scale=.5, vertices/.style={draw, fill=black, circle, inner sep=1.5pt}]
\node [vertices] (a) at (-0.3,0) {};
\node [vertices] (b) at (0.3,0) {};

\end{tikzpicture}
}

\newcommand{\dpointh}{
\begin{tikzpicture}[scale=.5, vertices/.style={draw, fill=black, circle, inner sep=1.5pt}]
\node [vertices] (a) at (-0.3,0) {};
\node [inner sep=2pt] (b) at (0,0.5){};
\node [vertices] (c) at (0.3,0) {};
\node (bb) at (0.2,0.2) {};

\foreach \to/\from in {a/b, c/b}
\draw (\to)--(\from);
\end{tikzpicture}
}

\newcommand{\cpointh}{
\begin{tikzpicture}[scale=.5, vertices/.style={draw, fill=black, circle, inner sep=1.5pt}]
\node [inner sep=2pt] (a) at (-0.5,0) {};
\node [vertices] (b) at (0,0.5) {};
\node (b1) at (-0.1,0.3) {};
\node (b2) at (0.1,0.3) {};

\foreach \to/\from in {b/a, b/c}
\draw (\to)--(\from);
\end{tikzpicture}
}

\newcommand{\linlh}{
\begin{tikzpicture}[scale=.5, vertices/.style={draw, fill=black, circle, inner sep=1.5pt}]
\node [inner sep=2pt] (a) at (-0.5,0) {};
\node [vertices] (b) at (0,0.7) {};
\node (b1) at (-0.1,0.5) {};
\node [vertices] (c) at (0.5,0) {};

\foreach \to/\from in {b/a, b/c}
\draw (\to)--(\from);
\end{tikzpicture}
}

\newcommand{\poly}{
\begin{center}
  \begin{tikzpicture}[dot/.style={draw,fill,circle,inner sep=2.3pt},scale=2]
\draw [help lines, white] (-1.5,-1.5) grid (1.5,1.2);


\draw [fill=blue, opacity=0.3] (0,1)--(.79,.63)--(-.79,.63)--(0,1);
\draw [fill=red, opacity=0.3] (.79,.63)--(-.43,-.9)--(-.79,.63)--(.79,.63);
\draw [fill=blue, opacity=0.3] (.79,.63)--(.43,-.9)--(-.43,-.9)--(.79,.63);
\draw [fill=blue, opacity=0.3] (-.43,-.9)--(-.79,.63)--(-.98,-.23)--(-.43,-.9);
\draw [fill=red, opacity=0.3] (.79,.63)--(.98,-.23)--(.43,-.9)--(.79,.63);

  \foreach \l [count=\n] in {0,1,2,3,4,5,6} {
    \pgfmathsetmacro\angle{90-360/7*(\n-1)}
      \node[dot] (n\n) at (\angle:1) {};
      \fill (\angle:1)  circle (0.09);
      }

\fill[red] (90:1) circle (0.065);
\fill[yellow] (141.43:1) circle (0.065);
\fill[red] (192.86:1) circle (0.065);
\fill[green] (244.29:1) circle (0.065);
\fill[red] (295.72:1) circle (0.065);
\fill[yellow] (347.15:1) circle (0.065);
\fill[blue] (398.58:1) circle (0.065);
      
  \foreach \l [count=\n] in {0,1,2,3,4,5,6} {
    \pgfmathsetmacro\angle{90-360/7*(\n-1)}
      \node (m\n) at (\angle:1.3) {$\n$};
  }
  \draw[thick] (n1) -- (n2) -- (n3) -- (n4) -- (n5) -- (n6) -- (n7) -- (n1);
  \draw[thick] (n2)--(n7)--(n5)--(n2)--(n4);
\end{tikzpicture}
\end{center}
}


\newcommand{\pointli}[2]
{ \raisebox{-2mm}{
\begin{tikzpicture}[scale=.5, vertices/.style={draw, fill=black, circle, inner sep=1.5pt}]
\coordinate (b) at (0,1) {};
\node [vertices, label=right: {$\scriptstyle{#1}$}] (c) at (0,0) {};
\coordinate [label=right: {$\scriptstyle{#2}$}] (a) at (-0.1,0.5) {};
\draw (c)--(b);
\end{tikzpicture}
}}

\newcommand{\pointe}[2]
{ \raisebox{-2mm}{
\begin{tikzpicture}[scale=.5, vertices/.style={draw, fill=black, circle, inner sep=1.5pt}]
  \coordinate (b) at (0,0) {};
  \coordinate (a) at (0,1) {};
\node [vertices, label=right:   {$\scriptstyle{#1}$}] (c) at (0.5,0){};
\coordinate [label=right: {$\scriptstyle{#2}$}] (d) at (-0.1,0.5){};
\draw (b)--(a);
\end{tikzpicture}
}}

\newcommand{\trekantl}[6]
{ \raisebox{-2mm}{
\begin{tikzpicture}[scale=0.8, vertices/.style={draw, fill=black, circle, inner sep=1.5pt}]
  \node [vertices, label=left:   {$\scriptstyle{#1}$}] (a) at (0,0) {};
  \node [vertices, label=above:   {$\scriptstyle{#2}$}] (b) at (0.65,1) {};
  \node [vertices, label=right:   {$\scriptstyle{#3}$}] (c) at (1.3,0) {};
\coordinate [label=above: {$\scriptstyle{#4}$}] (d) at (0.2,0.4){};
\coordinate [label=above: {$\scriptstyle{#5}$}] (e) at (0.65, -0.1){};
\coordinate [label=above: {$\scriptstyle{#6}$}] (f) at (1.1,0.4){};
\draw (a)--(b)--(c)--(a);
\end{tikzpicture}
}}

\newcommand{\trekantll}{
  \begin{tikzpicture}[scale=.5, vertices/.style={draw, fill, circle, inner sep=1.5pt}]

\draw [fill=blue, opacity=0.5] (-0.5,0)--(0,0.7)--(0.5,0)--(-0.5,0);

\node [vertices] (a) at (-0.5,0) {};
\node [vertices] (b) at (0,0.7) {};
\node [vertices] (c) at (0.5,0) {};
\node [vertices] (d) at (1,0.7) {};

\foreach \to/\from in {a/b, b/c, c/a, b/d, c/d}
\draw (\to)--(\from);
\end{tikzpicture}
}

\newcommand{\treell}{
  \begin{tikzpicture}[scale=.5, vertices/.style={draw, fill, circle, inner sep=1.5pt}]


\node [vertices] (a) at (-0.5,0) {};
\node [vertices] (b) at (0,0.7) {};
\node [vertices] (c) at (0.5,0) {};
\node [vertices] (d) at (1,0.7) {};

\foreach \to/\from in {a/b, b/c, c/a, b/d, c/d}
\draw (\to)--(\from);
\end{tikzpicture}
}

\newcommand{\treeltl}{
  \begin{tikzpicture}[scale=.5, vertices/.style={draw, fill, circle, inner sep=1.5pt}]


\node [vertices] (a) at (-0.5,0) {};
\node [vertices] (b) at (0,0.7) {};
\node [vertices] (c) at (0.5,0) {};
\node [vertices] (d) at (1,0.7) {};

\foreach \to/\from in {a/b, a/d, c/a, b/d, c/d}
\draw (\to)--(\from);
\end{tikzpicture}
}

\newcommand{\treel}{
  \begin{tikzpicture}[scale=.5, vertices/.style={draw, fill, circle, inner sep=1.5pt}]


\node [vertices] (a) at (-0.5,0) {};
\node [vertices] (b) at (0,0.7) {};
\node [vertices] (c) at (0.5,0) {};
\node (d) at (-1.5,0) {};

\foreach \to/\from in {a/b, b/c, c/a, a/d}
\draw (\to)--(\from);
\end{tikzpicture}
}

\newcommand{\edgell}{
  \begin{tikzpicture}[scale=.5, vertices/.style={draw, fill, circle, inner sep=1.5pt}]


\node [vertices] (a) at (0,0) {};
\node [vertices] (b) at (0,1) {};
\node (c) at (-0.8,0.1) {};
\node (d) at (-0.8,0.9) {};

\foreach \to/\from in {a/b, b/d, a/c}
\draw (\to)--(\from);
\end{tikzpicture}
}

\newcommand{\veel}{
  \begin{tikzpicture}[scale=.5, vertices/.style={draw, fill, circle, inner sep=1.5pt}]


\node [vertices] (a) at (-0.5,0) {};
\node [vertices] (b) at (0,0.7) {};
\node [vertices] (c) at (0.5,0) {};
\node (d) at (-0.3,0.9) {};

\foreach \to/\from in {a/b, a/c}
\draw (\to)--(\from);
\end{tikzpicture}
}

\newcommand{\bbox}{
  \begin{tikzpicture}[scale=.5, vertices/.style={draw, fill, circle, inner sep=1.5pt}]

\draw [fill=blue, opacity=0.5] (-0.5,0)--(0.5,0)--(0,0.5)--(-0.5,0);
\draw [fill=blue, opacity=0.5] (-0.5,1)--(0.5,1)--(0,0.5)--(-0.5,1);

\node [vertices] (a) at (-0.5,0) {};
\node [vertices] (b) at (0,0.5) {};
\node [vertices] (c) at (0.5,0) {};
\node [vertices] (d) at (-0.5,1) {};
\node [vertices] (e) at (0.5,1) {};

\foreach \to/\from in {a/c, c/e, b/a, b/e, e/d, d/a, b/c, c/a, c/d, c/d}
\draw (\to)--(\from);
\end{tikzpicture}
}

\newcommand{\bboxt}{
  \begin{tikzpicture}[scale=.5, vertices/.style={draw, fill, circle, inner sep=1.5pt}]

\draw [fill=blue, opacity=0.5] (-0.5,0)--(0,0.5)--(-0.5,1)--(-0.5,0);
\draw [fill=blue, opacity=0.5] (0.5,1)--(0,0.5)--(0.5,0)--(0.5,1);

\node [vertices] (a) at (-0.5,0) {};
\node [vertices] (b) at (0,0.5) {};
\node [vertices] (c) at (0.5,0) {};
\node [vertices] (d) at (-0.5,1) {};
\node [vertices] (e) at (0.5,1) {};

\foreach \to/\from in {a/c, c/e, b/a, b/e, e/d, d/a, b/c, c/a, c/d, c/d}
\draw (\to)--(\from);
\end{tikzpicture}
}


\section{Introduction}
\label{sec:intro}

We give thirty-two bialgebras of hypergraphs, or eight quartets of
bialgebras. All bialgebras have commutative multiplication.
Five of these quartets are genuinely distinct. Three of these five are
non-cocommutative, and the co-opposite of these give the last three
quartets.

Several pairs of these quartets come as four double bialgebras.
This latter notion was recently introduced and their general theory developed
by L.~Foissy \cite{Fo22}. A double bialgebra is the same as two
cointeracting bialgebras, where the underlying algebras are the same.

Foissy shows \cite{Fo22} that such a double bialgebra, with the comodule
bialgebra a Hopf algebra, comes with
a unique morphism to the double bialgebra structure naturally defined on the polynomial
ring $\QQ[x]$. In our case, i.e., for hypergraphs, we get associated
to any hypergraph
a {\it quartet} of polynomials in $\QQ[x]$. One of these
is the chromatic polynomial of the hypergraph  \cite{Doh,Hel},
which generalizes the classical chromatic polynomial of a graph.

Graph polynomials is a large subject. To a graph one may
associate many polynomials \cite{Fa, FoChrom}. However the chromatic
polynomial takes a distinctive position. It is certainly the most studied
\cite{Sa}. It is also canonical, strikingly established
by Foissy in \cite{FoChrom}, where he shows there is a {\it unique} double
bialgebra homomorphism from the double bialgebra of graphs to the double
bialgebra on $\QQ[x]$, and the image of a graph is its chromatic polynomial.

We generalize this double bialgebra to hypergraphs, but in this setting
there naturally appear three other double bialgebras, each with a unique
double bialgebra morphism to $\QQ[x]$, so that we get four canonical
polynomials associated to a hypergraph.

To give a taste of how this refines the study of hypergraphs and graphs,
recall that the chromatic polynomial of graphs does not distinguish trees
with the
same number of vertices. The four chromatic polynomials of hypergraphs will
however typically be quite distinct on trees with the same number of vertices.

\medskip
For a finite set $V$ let $P(V)$ be its power set, i.e., the set of all subsets
of $V$. A hypergraph is simply a map of sets 
$$
	h : E \pil P(V),
$$ 
where the elements of $E$ are called edges, and those of $V$ vertices. Such a hypergraph
induces three other hypergraphs, thus giving a quartet of hypergraphs:

\begin{itemize}
\item A {\it{dual hypergraph}} 
$$
	h^d : V \pil P(E),
$$ 
by considering for each vertex $v \in V$ all the edges containing $v$.
\item A {\it{complement hypergraph}} 
$$
	h^c : E \pil P(V),
$$ 
where $h^c(e) =  V \backslash h(e)$.
\item A {\it{dual complement hypergraph}} 
$$
	h^{cd} : V \pil P(E),
$$ 
by taking the dual of $h^c$ or equivalently the complement of $h^d$.
\end{itemize}

Let $H$ be the vector space with basis all isomorphism classes of finite
hypergraphs.
We first consider the restriction bialgebra
$(H, \mu,\Delta, \eta, \epsilon_\Delta)$ on hypergraphs, which 
straightforwardly generalizes the restriction bialgebra on graphs
introduced by W.~Schmitt \cite{Schmitt}.
Let $H^\circ \sus H$ be the subspace generated by the hypergraphs 
with no empty edges. It becomes a sub-bialgebra of the above. 
We then exhibit an extraction-contraction bialgebra 
$(H^\circ, \mu,\delta, \ben_\delta)$ such that $(H,\mu,\Delta)$ is a comodule
bialgebra over $(H^\circ, \mu, \delta)$. So
$(H^\circ, \mu,\Delta,\delta)$
becomes a double bialgebra in the sense of Foissy \cite{Fo22}.

Dualization and complementation are involutions on $H$. This transports
the double bialgebra $H^\circ$ to three other double bialgebras, resulting in
a quartet of double bialgebras:
\[ 
	H^\circ, \quad H^d, \quad H^c, \quad H^{cd}. 
\]

Each of them comes with a unique {\it double bialgebra} morphism to $\QQ[x]$.
We may extend 
each of these morphisms to a (single) bialgebra morphism from the restriction
bialgebra $(H, \mu, \Delta)$ to $\QQ[x]$. So for each
hypergraph $h$ we get four corresponding polynomials
\[ 
	\chi_h(x), \quad \chi_h^d(x), \quad \chi_h^c(x), \quad \chi_h^{cd}(x). 
\]

The first polynomial, $\chi_h(x)$, is the chromatic polynomial of
the hypergraph $h$.
The value $\chi_h(k)$  counts the number of colorings of the vertices of
$h$ with $k$ colors, such that {\it no edge is monochromatic}.

\begin{example}
Consider the hypergraph $h$ with one edge $E=\{e\}$ and $|V|=n$ vertices, the
edge containing all $n$ vertices, i.e., $h(e)=V$. The hypergraph chromatic
polynomial is then
$\chi_h(x) = x^n - x$.
We also note that the chromatic polynomial of a hypergraph $h$  
vanishes whenever $h$ has an edge with exactly one vertex.
\end{example}

The notion of chromatic polynomial for hypergraphs, seems to date back to
the article \cite{Hel}. It has further been considered in various
articles, like \cite{Doh,Tom, ZD}.
The other polynomials above also have interpretations in terms of colorings,
but have to our knowledge, at least in a systematic setting,
not been considered before.

In the last section we give three more quartets of bialgebras.
The first comes from the bialgebras of graphs and hypergraphs
introduced in \cite{AA} Subsections 3.1 and 20.1.
The coproduct $\Delta^\prime$
is now no longer co-commutative. It is (for the first bialgebra)
restriction of edges on the first factor,
but descent of edges on the second factor. 
The second quartet comes from a coproduct $\Delta^{\prime \prime}$
which is descent of edges on both factors. 
So this is again a co-commutative coproduct. 
The third quartet comes from an extraction-contraction bialgebra
introduced recently by L.Foissy \cite{Fo-Hyp}. It cointeracts with
the above descent-descent bialgebra.

\medskip Summing up we have in all thirty-two bialgebras for hypergraphs:

\begin{itemize}
\item[$\bullet \, 4$:] A quartet of restriction-restriction bialgebras,
 (the base case is classical), Section \ref{sec:hyp},
\item[$\bullet \, 8$:] A quartet of cointeracting extraction-contraction bialgebras
  and its co-opposite, (the coproduct is not co-commutative),
  Section \ref{sec:ext-con}
\item[$\bullet \, 8$:] A quartet of restriction-descent bialgebras
  and its co-opposite (base case M. Aguiar and F. Ardila \cite{AA}),
  Subsection \ref{subsec:res-descent},
\item[$\bullet \, 4$:] A quartet of descent-descent bialgebras,
  Subsection \ref{subsec:des-des},
\item[$\bullet \, 8$:] A quartet of cointeracting extraction-contraction
  bialgebras and its co-opposite (base case L.~Foissy \cite{Fo-Hyp}),
  Subsection \ref{subsec:res-contract2}
\end{itemize}

The last two quartets give four double bialgebras, and so four new
associated polynomials. For the base case this is essentially
the classical chromatic polynomial of a graph, \cite{Fo-Hyp}. It is the
chromatic polynomial of the graph obtained by replacing each hyperedge
with a complete graph on its set of vertices. Nevertheless, the quartet
gives four polynomials instead of the single classical one. For
a hypergraph, and the graph obtained by the replacement above,
the three polynomials not from the base case are usually quite distinct.

\medskip

The organization of the article is as follows.
Section \ref{sec:hyp} recalls basics around hypergraphs, and
the notions of dualization $d$ and complementation $c$. We give the restriction
bialgebra of hypergraphs,
and the three other bialgebras derived using $d$, $c$ and $cd$. Section
\ref{sec:coint} recalls the notion of cointeracting bialgebras, and
double bialgebras from \cite{Fo22}.
Section \ref{sec:ext-con}
introduces the extraction-contraction bialgebra on hypergraphs.
We show that the restriction bialgebra of Section \ref{sec:hyp}
is a comodule bialgebra over this bialgebra.
Section \ref{sec:poledge}
considers the maps from the double bialgebras of hypergraphs
to $\QQ[x]$ and computes the associated polynomial for the simplest
of hypergraphs: i.~the hypergraphs with no edges, ii.~the discrete hypergraphs,
where each edge is paired with exactly one vertex, and iii.~the hypergraphs with
only one edge containing all vertices.
Section \ref{sec:chrompol}
shows that the associated polynomial of hypergraphs is the chromatic
polynomial, counting non-monochromatic colorings.
Section \ref{sec:quartpol}
gives examples of the quartet of polynomials for various
hypergraphs. In Section \ref{sec:newq}
we introduce first another quartet of bialgebras
of hypergraphs, derived from the bialgebras of graphs and hypergraphs
introduced in \cite{AA}.
Moreover we ask if there are cointeracting bialgebras associated to these
bialgebras. Secondly we give two more bialgebras, establishing in 
total fourteen bialgebras on the vector space $H$ of isomorphism classes of
finite hypergraphs.

\medskip
\noindent {\it Acknowledgements:} The second author thanks NTNU for
hosting a longer stay where the initial phases of this work was done.
He received support from Lorentz Meltzers h{\o}yskolefond.

We thank L.~Foissy for bringing to our attention some inaccuracies
in the first version of this article, concerning the cointeractions of
bialgebras


\section{Bialgebras of hypergraphs}
\label{sec:hyp}

We give our notion of hypergraphs, and the hypergraphs one may derive
using the process of dualization and complementation.
We give the restriction coalgebra on hypergraphs, and the three
other coalgebras one gets by transporting this using dualization and
complementation. We list examples of the coproducts.
There are also two commutative products. In total we get four bialgebras
of hypergraphs.


\subsection{Hypergraphs, relations, and bipartite graphs}

 For a set $S$ denote by $P(S)$ its
 power set, consisting of all subsets of $S$. The complement
 of a subset $T \sus S$ is written $S \backslash T$, but when
 $S$ is understood, we simply write $T^c$. The subset $T$ may be
 identified with a map $S \pil \{0,1\}$ by sending the elements of $T$
 to  $1$ and other elements to $0$. Thus 
 \[ 
 	P(S) = \Hom(S, \{0,1\}).
\]

\begin{definition}
Let $V$ and $E$ be sets. A {\it hypergraph} is a map $h : E \pil P(V)$.
The elements of $V$ and $E$ are {\it vertices} and {\it edges}, respectively.
\end{definition}

For an edge $e \in E$, the subset $h(e) \sus V $ is the set of vertices 
of $e$. By abuse of notation we may write $e \sus V$.
We have {\it no restrictions} on the map $h$.
We allow edges to have an empty set of vertices.
We allow different edges to have the same set of vertices, or more
generally that their vertex sets may be related by inclusion.
A hypergraph is thus an element in
\begin{equation} 
\label{eq:hyp-EV} 
	\Hom(E, \Hom(V, \{0,1\})) 
	= \Hom(E \times V, \{0,1\}) 
	= \Hom(V, \Hom(E, \{0,1\})). 
\end{equation}
By the middle part above, our general notion of a hypergraph is
simply equivalent to a subset of $E \times V$, or a relation between $E$
and $V$. We could thus equally well (and maybe more appropriately) have
called this a relation. However the connotations suggested by
vertices and (hyper)edges will be natural, so we use this.

\begin{remark}
A relation between $E$ and $V$ also identifies as a
bipartite graph with vertices $E \cup V$.
So there is the ``paradox'':
\[ 
	\text{bipartite graphs } \subset \text{ graphs }
  	\subset \text{ hypergraphs } = \text{ bipartite graphs}.
\]
\end{remark}

\begin{remark}
  A map $E \pil \Hom(V,\{0,1\})$ can be considered a {\it promap}
  or {\it profunctor} $E \promap V$. This is a special case of a profunctor
  between partially ordered sets \cite{Fl}, or even of profunctors
  between categories \cite[Ch.4]{ACT}.
\end{remark}




\subsection{Derived hypergraphs}

By \eqref{eq:hyp-EV} above, we get a dual hypergraph
\[ 
	h^d : V \pil P(E), 
\]
where $V$ becomes the edges of the dual hypergraph, and $E$ the vertices.
There is a complementation map
\[ 
	P(V) \mto{c} P(V), \quad S \mapsto S^c = V\backslash S 
\]
which is an involution. Composing with $h$ we get a complement hypergraph
\[ 
	h^c : E \mto{h} P(V) \mto{c} P(V). 
\]

We may also take the dual of the complement $h^{cd} = (h^c)^d$ and the
complement of the dual $h^{dc} = (h^d)^c$. These are equal (see
Example \ref{ex:hyp-EV} below). The hypergraph $h$ induces four hypergraphs:
\[ 
	h, \quad h^d, \quad h^c, \quad h^{cd}. 
\]
And by \eqref{eq:hyp-EV} the hypergraph $h : E \pil P(V)$ is  
equivalent to a relation between $E$ and $V$, and so may be represented by
a $0,1$-matrix with rows indexed by $E$ and columns indexed by $V$.

\begin{example} \label{ex:hyp-EV}
  Let the hypergraph $h$ be given by the $0,1$-matrix below.
  The three other hypergraphs are then displayed in the same way.
  \begin{equation*}
h =   E \,  \overset{\large{V}}{\left [ \begin{matrix} 0 & 1 & 0 & 1 \\
        1 & 0 & 1 & 1 \end{matrix} \right ]}, \quad
 h^c = E \, \overset{V}{\left [ \begin{matrix} 1 & 0 & 1 & 0 \\
      0 & 1 & 0 & 0 \end{matrix} \right ]}, \quad
h^d =  V \, \overset{E}{\left [ \begin{matrix} 0 & 1 \\  1 & 0  \\
      0 & 1 \\  1 & 1 \end{matrix} \right ]}, \quad
 h^{cd} = V \, \overset{E}{\left [ \begin{matrix} 1 & 0 \\ 0 & 1 \\
      1 & 0 \\  0 & 0 \end{matrix} \right ]} 
\end{equation*}
\end{example}


\subsection{The coalgebra}

For $U \sus V$, the {\it restriction} of edges to $U$ are the edges:
\[ 
	E_{|U} = \{ e \in E \, | \, e \sus U \}.  
\]

Let $\kk$ be a field.
Let $\Hy$ be the $\kk$-vector space generated by isomorphism
classes of hypergraphs $(E,V,h)$ where $V$ and $E$ are finite sets.
For $U \sus V$, the complement set is $U^c = V \backslash U$.
We have a coproduct $\Delta$ defined as follows (we omit
the maps $h$ as they are understood):
\[  
	(E,V) \overset{\Delta}{\longmapsto}
  \sum_{U \sus V} (E_{|U}, U) \te (E_{|U^c}, U^c). 
\]
There is a counit $\epsilon_\Delta$ on $\Hy$:
\[ 
	(E,V) \overset{}{\mapsto} \begin{cases} 1, & V = \emptyset \\
    0, & \text{ otherwise } \end{cases}. 
\]
This gives a coalgebra structure on $\Hy$, where we write $\een$
for the hypergraph $(\emptyset, \emptyset)$. 

\begin{example}
\begin{equation*}
  \cherry  \overset{\Delta} \mapsto
 \een\te \cherry  + 2 \, \point \te \linvert + \point \te \dpoint           
  + 2 \, \linvert \te \point + \dpoint \te \point
             + \cherry \te \ben
\end{equation*}
\end{example}
More examples are gathered in  Subsection \ref{subsec:hyp-ex}.


\subsection{The coalgebra from dualization}

Due to the symmetric situation of edges and vertices in a relation
we also have a dual coproduct.
Since $d$ is an involution, we get a coproduct
\[
	\Delta^d = (d \te d) \circ \Delta \circ d 
\]
on hypergraphs. Explicitly, for each $F \sus E$, let the
{\it restriction}
\[ 
	V_{|F} = \{ v \in V \, | \, \text{ the edges incident to  } v \text{ are in } F \}. 
  \]
Also let $F^c = E \backslash F$ be the complement.
The coproduct $\Delta^d$ is given by:
\[  
	(E,V) \mapsto \sum_{F \sus E} (F, V_{|F}) \te (F^c, V_{|F^c}), 
\]
and a counit $\epsilon^d_\Delta$: 
\[ 
	(E,V) \overset{}{\mapsto} \begin{cases} 1, & E = \emptyset \\
    0, & \text{ otherwise } \end{cases}. 
\]
This gives a second coalgebra structure on $\Hy$.


\subsection{The coalgebra from complementation}

The involution $c$ gives a coproduct
\[ 
	\Delta^c = (c \te c) \circ \Delta \circ c. 
\]
For a subset $U \sus V$, let the {\it link}
\[ 
	\lk_U E = \{ e \in E \, | \, h(e) \supseteq U \}
  = \{ e \in E \, | \, e \text{ is incident to every } u \in U \}. 
 \]
The complement $\Delta^c$ is then defined by
\[ 
	(E,V) \mapsto \sum_{U \sus V} (\lk_{U^c} E, U) \te (\lk_U E, U^c), 
\]
and the counit $\epsilon^c_\Delta$ is the same as $\epsilon_\Delta$:
\[ 
	(E,V) \overset{}{\mapsto} \begin{cases} 1, & V = \emptyset \\
    0, & \text{ otherwise } \end{cases}. 
\]


\subsection{The coalgebra from the dual-complement}

The involution $d \circ c$ gives a coproduct
\[ 
	\Delta^{cd}  = (d \circ c \te d \circ c) \circ \Delta \circ (d \circ c). 
\]
For a subset $F \sus E$, let the {\it core}
\[ 
  \co_F V = \cap_{f \in F} h(f) = \{ v \in V \, | \, v
  \text{ is incident to every } f \in F\}. 
 \]
The dual-complement $\Delta^{cd}$ is then defined by
\[ 
	(E,V) \mapsto \sum_{F \sus E} (F, \co_{F^c} V) \te (F^c, \co_F V), 
\]
and the counit $\epsilon^{cd}_\Delta$ is the same as $\epsilon^d_\Delta$:
\[ 
	(E,V) \overset{}{\mapsto} \begin{cases} 1, & E = \emptyset \\
    0, & \text{ otherwise } \end{cases}. 
\]



\subsection{Examples}
\label{subsec:hyp-ex}

Here are the various coproducts for the path graph with three vertices.
We write $\een$ for the hypergraph $(\emptyset,\emptyset)$.
\begin{align*}
  \cherry  & \overset{\Delta} \longmapsto
 \een \te \cherry  + 2 \, \point \te \linvert + \point \te \dpoint           
  + 2 \, \linvert \te \point + \dpoint \te \point
             + \cherry \te \een  \\
  \cherry & \overset{\Delta^d} \longmapsto \een \te \cherry +
            2 \, \pointl\te \pointl + \cherry \te \een \\
  \cherry  & \overset{\Delta^c} \longmapsto
             \een \te \cherry  + 2 \, \point \,  \edge \te \pointl \point  +
             \point \te \pointl \pointl           
  + 2 \, \pointl \point  \te \point \,  \edge + \pointl \pointl \te \point
             + \cherry \te \een  \\
     \cherry & \overset{\Delta^{cd}} \longmapsto \een \te \cherry +
               2 \, \pointl \point \te \pointl \point + \cherry \te \een
\end{align*}

Here are the coproducts for the hypergraph with one edge on three vertices:
  \begin{align*}
    \trekant & \overset{\Delta}{\longmapsto} \een \te \trekant
    + 3 \, \point \te \dpoint + 3 \, \dpoint \te \point + \trekant \te \een \\
    \trekant & \overset{\Delta^d}{\longmapsto} \een \te \trekant
               +  \trekant \te \een \\
    \trekant & \overset{\Delta^c}{\longmapsto} \edge \te \trekant
               + 3 \pointl \te \linvert + 3 \, \linvert \te \pointl +
    \trekant \te \edge \\
    \trekant & \overset{\Delta^{cd}}{\longmapsto} \trev \te \trekant
    +  \trekant \te \trev
  \end{align*}
  Note that for $\Delta^c$ we get an empty edge instead of $1$ in the
  first tensor term. This is due to the complement of the triangle edge being
  an empty edge.


\subsection{Relation between the coproducts}

There is the following relation between these coproducts. 

\begin{proposition} Let the indices $\ell$ and $k$ be either $c,d,cd$ or empty.
Write $\id$ for the identity map. Then:
\[ 
	(\id \te \Delta^\ell) \circ \Delta^k = \tau_{1,3} (\Delta^\ell \te \id) \circ \Delta^k. 
\]
\end{proposition}
Note that $\tau_{1,3}(a \otimes b \otimes c):=c \otimes b \otimes a$.

\begin{proof}
This is a formal consequence of $\Delta^\ell$ and $\Delta^k$ both being cocommutative. Using Sweedler notation
\begin{equation} \label{eq:rel-D1}
    h \overset{\Delta^k} \longmapsto h_{(1)} \te h_{(2)} = h_{(2)} \te h_{(1)}.
\end{equation}
Then $\id \te \Delta^\ell$ maps the $\Delta^k(h)=h_{(1)} \te h_{(2)}$ to (using cocommutativity)
\[ 
  	h_{(1)} \te h_{(21)} \te h_{(22)} = h_{(1)} \te h_{(22)} \te h_{(21)}.
\]
  Switching the first and third term, this is
\begin{equation} \label{eq:rel-D12}  
	h_{(21)} \te h_{(22)} \te h_{(1)}.
\end{equation}
On the other hand applying $(\Delta^\ell \te \id)$ to the right side of \eqref{eq:rel-D1}, i.e., $\Delta^k(h)=h_{(2)} \te h_{(1)}$, we get precisely \eqref{eq:rel-D12}.
\end{proof}


\subsection{Products and bialgebras}
\label{subsec:hyp-products}
We have a product $\mu$ on $\Hy$ by:
\[ 
	(E,V,h) \cdot (E^\prime, V^\prime,h^\prime)
  	= (E \sqcup E^\prime, V \sqcup V^\prime,h \sqcup h^\prime),
\]
and the unit being $1 = (\emptyset, \emptyset)$.
Write $\eta : \kk \pil H$ sending $1 \mapsto 1 = (\emptyset, \emptyset)$.
With this product, we have two bialgebra structures on $\Hy$:
\[ 
	(\Hy, \mu, \Delta, \eta, \epsilon_\Delta), \quad
        (\Hy, \mu, \Delta^d, \eta, \epsilon^d_\Delta). 
\]

There is also another product $\mu^c$ on $\Hy$ given by
$\mu^c = c \circ \mu \circ (c \te c)$: 
\[ 
	(E,V,h) \cdot (E^\prime, V^\prime,h^\prime)
  = (E \sqcup E^\prime, V \sqcup V^\prime,h \hat{\sqcup} h^\prime),
\]
where $h \hat{\sqcup} h^\prime$ sends $e \in E$ and $e^\prime \in E^\prime$ to respectively:
\[ 
	e \mapsto h(e) \cup V^\prime, \quad e^\prime \mapsto V \cup h^\prime(e^\prime). 
\]
The unit is again $\een = (\emptyset, \emptyset)$. With this product
we have another two bialgebra structures on $\Hy$, which we denote by
\[ 
	(\Hy, \mu^c, \Delta^c, \eta, \epsilon^c_\Delta), \quad
        (\Hy, \mu^c, \Delta^{cd}, \eta, \epsilon^{cd}_\Delta). 
\]

The bialgebras $(\Hy,\mu,\Delta) $ and $(\Hy, \mu^c, \Delta^c)$
are graded by cardinalities
of vertices.
On the other hand, the bialgebras $(\Hy, \mu, \Delta^d)$ and
$(\Hy, \mu^c, \Delta^{cd})$ are graded by cardinalities of edges.


\section{Cointeracting bialgebras}
\label{sec:coint}

We recall the notion of cointeracting bialgebras.
For more detail and examples, see the concise review of D.~Manchon \cite{Man}.
Let
\[ (A,\mu_A, \Delta_A, \eta_A, \epsilon_A), \quad
  (B,\mu_B, \delta_B, \eta_B, \epsilon_B) \]
be bialgebras over the field $\kk$. We suppose $A$ is a (left) comodule over $B$:
\[ \delta : A \pil B \te A, \] and want this to fulfill the following
conditions ($\id$ denotes identity maps):

\begin{itemize}
\item[1)] The counit $\eps_A : A \pil \kk$ is a comodule morphism:
  \begin{equation} \label{eq:dobi-counit}
    (\id_B \te \epsilon_A) \circ \delta = \eta_B \circ \epsilon_A.
    \end{equation}
  \item[2)] The coalgebra map $\Delta_A : A \pil A \te A$ is a
    comodule morphism:
    \begin{equation} \label{eq:dobi-comod}
(\id_B \te \Delta_A) \circ \delta = m_{13,2,4} \circ (\delta \te \delta)
\circ \Delta_A,
\end{equation}
where $m_{13,2,4}$ is multiplication in $B$ for the $1$'st and $3$'rd factor.
\item[3)] The unit $\eta : \kk \pil A$ is a comodule morphism, which amounts
  to $\delta(1_A) = 1_B \te 1_A$.
\item[4)] The multiplication $\mu_A : A \te A \pil A$ is a comodule morphism:
  \[ \delta \circ \mu_A = (\id_B \te \mu_A) \circ m_{13,2,4} \circ
    (\delta \te \delta). \]
\end{itemize}

\begin{remark}
The characters  of $A$, the "dual object" of $A$,
is then a monoid $M(A)$, whose
multiplication is the dual of the coproduct $\Delta_A$.
(If $A_\Delta$ is a Hopf algebra, then $M(A)$ is a group.)

The characters of $B$ also get a monoid structure $E(B)$ by the dual of
the coproduct $\delta_B$. This monoid acts as endomorphisms on
the monoid of characters $M(A)$ by the duals of $\delta$ and 
\eqref{eq:dobi-comod}.
\end{remark}

When $A = B$ and:
\begin{itemize}
  \item $\mu_A = \mu_B$, so they are the same  as algebras, and
  \item  $\delta = \delta_B$,
\end{itemize}
following
L.~Foissy \cite{Fo22}, we call $(B,\mu, \Delta_B, \delta_B)$ with its two
counits $\eps_\Delta$ and $\vep_\delta$ a {\it double bialgebra}. 
Then 3) and 4) above, follow from $B$ being a bialgebra. 

\begin{remark}
  By \cite[Cor.2.4]{Fo22}, for a double bialgebra,
  if $(B, \mu, \Delta)$ is a Hopf algebra, the
multiplication $\mu$ is commutative.
\end{remark}

\begin{remark}
  The first example of a double bialgebra seems to be
  \cite{CEM}, where $B_\Delta$ is the Connes--Kreimer Hopf algebra,
  \cite{CK}. The coproduct $\Delta$ is given by admissible
  cuts in non-planar rooted trees, and $\delta$ is given by partitioning trees
  into subtrees and contracting these ({\it{extraction-contraction}}). 
  Quasi-shuffle double bialgebras occur in \cite{EM}, and
  in a slightly more general setting in \cite[Sec.1.3]{Fo22}.
  \cite{FFM} has double bialgebras for finite topologies
  (which may be identified with finite preorders).
  For more details and examples see \cite{Man}.
\end{remark}

In the following 
our focus is double bialgebras,
but we would like to make slightly
more general room. We suppose $B \sus A$ as vector spaces, and:

\begin{itemize}
  \item The multiplication
    restricts $m_A|_B = m_B$.
  \item The comodule morphism $\delta$ restricts to the coproduct $\delta_B$.
  \end{itemize}

  In our case $A$ is the vector space $H$ spanned by (isomorphism classes) of
  finite hypergraphs. For $B$ we restrict to various
  subclasses of hypergraphs.
  With each of these restrictions $(B,\mu_B, \Delta_B)$ becomes a
  {\it connected}
  sub-bialgebra of $(A, \mu_A, \Delta_A)$ and so
  a Hopf algebra. Furthermore $(B,\mu_B, \Delta_B, \delta_B)$ is a double
  bialgebra.
  In each case there will also be a sub-bialgebra $C$ of $A$
  such that
  $A \iso B \te C$.


\section{Extraction-contraction bialgebras}
\label{sec:ext-con}

We introduce the extraction-contraction bialgebra for hypergraphs.
This generalizes the extraction-contraction bialgebra for graphs.
However the algebra for graphs seems to be done only for simple
graphs in the literature. In the coproduct one then sums over
certain partitions of the vertex sets. But to get this for hypergraphs
we must rather sum over (nearly) all subsets of edges. This
point of view occurs in \cite{DFM} considering Tutte polynomials for
minor systems, in particular for graphs \cite[Section 4.2]{DFM},
or in \cite{KMT} considering Tutte polynomials for species, which
connects to bialgebras via the Fock functor, \cite[Chap.15]{AM}. 
See also \cite{Man12}.

We present an extraction-contraction bialgebra in cointeraction
with the restriction bialgebra on hypergraphs.
Via dualization and complementation we transport this to four double
bialgebras on hypergraphs.



\subsection{Connectedness and contractions} 
\label{subsec:EC-quot}

For each vertex set $W$, there is the distinguished {\it discrete} hypergraph,
whose edges $E = W$ , i.e., are the sets $\{w\}$ for $w \in W$. 
The hypergraph $(W,W)$ is given by the canonical map $W \pil P(W)$.

Given a map $\phi : V \pil W$ we get a map $P(V) \mto{P\phi} P(W)$
sending $S \mapsto \phi(S)$.
Assume now $\phi$ is surjective and there is a map $\psi : E \pil W$
giving a commutative diagram:
\begin{equation} \label{eq:EC-EVmor}
  \xymatrix{ E \ar[r]^{h} \ar[d]_{\psi} & P(V) \ar[d]^{P\phi} \\
    W \ar[r] & P(W)}.
  \end{equation}
Note that in such a diagram, the map $\psi$ is uniquely determined
by $\phi$. 

\begin{definition} For a hypergraph $(E,V)$ let $E^*$ be the non-empty
  edges in $E$. 
  The hypergraph $(E,V)$ has {\it connected vertex set}
  if the only such diagram for $(E^*,V)$ 
with $V \ppil W$ surjective is when  $W$ is a single point.
This means the equivalence relation on $V$ generated by $u \sim u^\prime$
if $u$ and $u^\prime$ are on a common edge, only has one equivalence class,
the whole of $V$. 
\end{definition}

That $\phi : V \ppil W$ gives a diagram \eqref{eq:EC-EVmor}
for $(E^*,V)$ with
a discrete hypergraph $(W,W)$, means that the vertex set of each
edge $e \in E^*$ maps
to a single point in $W$ (depending on $e$). 
Let $(E^*_j,V_j), j \in J$ be the connected components of $(E^*,V)$ for the
equivalence relation above. Let $V^\prime = \cup_{j \in J} V_j$ and
$W = J \cup (V \backslash V^\prime)$. (Note that $V \backslash V^\prime$
are the vertices which are not incident to any edge.)
This gives a surjection $V \ppil W$, and
it gives a commutative diagram \eqref{eq:EC-EVmor}
for $(E^*,V)$ and the discrete hypergraph
$(W,W)$, which is initial among such diagrams. 

\medskip

For each subset $F \sus E^*$, we get a sub-hypergraph $(F,V)$
with connected components $(F_i,U_i), i \in I$. Let  $U = \cup_{i \in I} U_i$,
the {\it vertex support} of $(F,V)$,
and let $V/F = I \cup (V\backslash U)$ be the {\it contraction} of
$V$ by $F$. 
This gives a surjection $V \ppil V/F$, and a commutative diagram
with \eqref{eq:EC-EVmor} for $(F,V)$ and 
$(V/F,V/F)$, which is initial among
all commutative diagrams to discrete hypergraphs.


\subsection{Coproduct $\delta$} 
\label{subsec:EC-d1}
Let $H^\circ \sus H$ be generated by the hypergraphs where each edge
has a non-empty vertex set. That is hypergraphs $(E,V)$ with $E = E^*$,
or equivalently $E_{|\emptyset} = \emptyset$.
Note that $H^\circ \sus H$ is actually
a sub-bialgebra of $(H,\mu,\Delta)$ as the coproduct $\Delta$ restricts to
$H^\circ$.

We can now make $H$ a comodule over $H^\circ$ by the coproduct:
\begin{equation} \label{eq:extcon-delta}
  \delta : H \pil H^\circ \te H, \quad
  (E,V) \overset{\delta}{\mapsto} \sum_{F \sus E^*} (F,V) \te (F^c,V/F).
  \end{equation}
The pair $(F^c,V/F)$ is a hypergraph by the composition
\[ F^c \sus E \pil P(V) \pil P(V/F). \]
Note that all edges with empty vertex sets go into $F^c$ in the right
tensor term, above.

\begin{example}
 \begin{equation*}
    \cherry \overset{\delta}{\mapsto} \trev \te \cherry  +
    2 \, \linP \te \linvert + \cherry \te \bullet
  \end{equation*}

The next example is a single edge with three vertices:
  \begin{equation*}
    \trekant \overset{\delta}{\mapsto} \trev \te \trekant  
     + \trekant \te \bullet
    \end{equation*}

Now we consider two edges, one with two vertices, and one with three vertices:
 \begin{equation*}
    \trekantKant \overset{\delta}{\mapsto} \trev \te \trekantKant
     + \linP \te \linvert + \trekant \te \pointl + \trekantKant \te \bullet
    \end{equation*}
    Note the vertex with a single edge attached. Often in the literature
    an edge with single vertex 
    is displayed as a loop, but here we do as above.
\end{example}
    

The comodule coproduct $\delta$ restricts to a coproduct
$\delta^\circ : H^\circ \pil H^\circ \te H^\circ$. 
There is a counit $\vep_\delta$:
\[ (E,V) \overset{}{\mapsto} \begin{cases} 1, & E = \emptyset \\
    0, & \text{ otherwise } \end{cases}. 
\]
which is the same as the map
$\epsilon^d_\Delta$ restricted to $H^\circ$.

\begin{proposition}
  $(H^\circ,\mu, \delta^\circ, \eta, \vep_\delta)$ is a bialgebra.
\end{proposition}

\begin{proof}
  The essential thing is to prove that $\delta$ is coassociative. Letting
  the edge sets $A, B$ be disjoint subsets of $E = E^*$,
  this amounts to show that the map
  $V \pil V/(A \cup B)$ identifies as the composition
  $V \pil V/A \pil (V/A)/B$, which is clear.

  By looking at the terms of \eqref{eq:extcon-delta} when $F = \emptyset$ and
  when $F = E^* = E$ we also see:
  \[ (\vep_\delta \te \id) \circ \delta = (\id \te \vep_\delta) \circ
    \delta = \id. \]
\end{proof}


\begin{theorem} \label{thm:EC-cointer1}
  The bialgebra $(\Hy,\mu, \Delta, \eta, \eps_\Delta)$
  is a comodule bialgebra over
  the bialgebra $(H^\circ,\mu, \delta^\circ, \eta, \vep_\delta)$
  via $\delta$ in \eqref{eq:extcon-delta}.

  In consequence  $(H^\circ, \mu, \Delta^\circ, \delta^\circ)$
  is a double bialgebra.
\end{theorem}


\begin{proof}[Proof of Theorem \ref{thm:EC-cointer1}.]
The coproducts are:
\begin{align*}  (E,V) & \overset{\Delta}{\mapsto}
                        \sum_{V^\prime \sus V} (E_{|V^\prime }, V^\prime)
                        \te (E_{|V \backslash V^\prime }, V \backslash V^\prime), \\
  (E,V) & \overset{\delta}{\mapsto} \sum_{F \sus E^*} (F,V) \te (F^c,V/F).
\end{align*}
We must show properties 1)-4) of Section \ref{sec:coint}.
Property 3) is easy, and 4) also, as multiplication is disjoint
union of hypergraphs. 

\noindent{\bf Property 1).}
For $(E,V)$ with $V \neq \emptyset$, both sides of $\eqref{eq:dobi-counit}$
vanish. When $V = \emptyset$ we have:
\[ 
	\xymatrix{ (E,\emptyset) \ar[rr]^{\epsilon_\Delta} \ar[d]_{\delta}
    	& & 1 \ar[d]^{\eta} \\
    	(\emptyset, \emptyset) \te (E,\emptyset)
    	\ar[rr]^{\id \te \epsilon_\Delta}&  & \ben_{H^\circ} = (\emptyset, \emptyset). }
\]



\noindent{\bf Property 2).}

\noindent 1.
Let us describe the composition $(\id \te \Delta) \circ \delta$.
Using the notation at the end of the last subsection, we may write
$V/F = I \cup (V\backslash U)$, where $U$ is the support of $(F,V)$.
Now let $S \sus I$ and $R \sus V \backslash U$,
so we get decompositions $I = S \sqcup S^c$ and
$V \backslash U = R \sqcup R^c$.
Further we define $U_S = \cup_{i \in S} U_i$ and
$U_{S^c} = \cup_{i \in S^c} U_i$. 
As $S$ and $R$ varies, the composition $(\id \te \Delta) \circ \delta$
is the sum of terms
\begin{equation} \label{eq:EC-Dd} 
  (F,V) \te (F^c_{|S \cup R}, S \cup R) \te (F^c_{|S^c \cup R^c}, S^c \cup R^c).
  \end{equation}

\medskip
\noindent 2. Let us describe the composition $(\delta \te \delta) \circ \Delta$.
  For $F^1 \sus E_{|V^\prime}$, let its components be $(F^1_i,U_i), i \in S^1$
  and $U^1 = \cup_{i \in S^1} U_i$. Similarly for
  $F^2 \sus E_{|V \backslash V^\prime}$ we get $S^2$ and $U^2$.
  Write $R^1 = V^\prime \backslash U^1$ and
  $R^2 = (V \backslash V^\prime) \backslash U^2$. 
  As $F^1$ and $F^2$ varies the composition $(\delta \te \delta) \circ \Delta$
 then  sends $(E,V)$ to a sum of terms:
\begin{equation} \label{eq:EC-mdd}
  (F^1,V^\prime ) \te ((E_{|V^\prime} \backslash F^1), S^1 \cup R^1)
\te (F^2, V\backslash V^\prime) \te
((E_{|V \backslash V^\prime} \backslash F^2), S^2 \cup R^2).
\end{equation}

\medskip
\noindent 3.
The bialgebras are in cointeraction if the composite maps are equal:
\[ 
	m_{13,2,4} \circ (\delta \te \delta) \circ \Delta 
	= (\id \te \Delta) \circ \delta, 
 \]
which amounts to showing that the maps giving terms 
\eqref{eq:EC-Dd} and \eqref{eq:EC-mdd} coincide
after the multiplication of the first and third entry in the tensor product in \eqref{eq:EC-mdd}.

\medskip
\noindent 4.
The data in part 1 are $F,S, R$ and derived data $S^c, R^c, U, U_S, U_{S^c}$.
The data in part 2 are
$F^1, F^2, V^\prime$ and derived data $R^1, R^2, S^1, S^2$.
Let us see how these correspond to each other.
Given the data in part 1, we let
\[ 
	U^1 = U_S, \quad F^1 = F_{|U_S}, \quad S^1 = S, \quad R^1 = R.
\]
Note that $V^\prime = U^1 \cup R^1$. Also
\[ 
	U^2 = U_{S^c}, \quad F^2 = F_{|U_{S^c}}, \quad S^2 = S^c, \quad R^2 = R^c. 
\]
Then $V \backslash V^\prime = U^2 \cup R^2$. Note also 
\[ 
	R^1 \cup R^2 = V \backslash (U^1 \cup U^2) = V \backslash U. 
\]
Finally note that $F_1 \cup F_2 = F$ since $U_S$ and $U_{S^c}$
decompose $(F,V)$ into two components.

\medskip

\noindent 5. Conversely, given the data in part 2,
we get
\[ 
	S = S^1, S^c = S^2, R = R^1, R^c = R^2, \quad F = F^1 \cup F^2. 
\]
Hence the two maps with terms in \eqref{eq:EC-Dd} and \eqref{eq:EC-mdd}
identify (after multiplying the first and last terms in the latter).
\end{proof}

\begin{remark}
The double bialgebra for graphs was considered by L.~Foissy in reference \cite{FoChrom}, and before that by W.~Schmitt \cite{Schmitt} (but Schmitt does not state the cointeraction property). See also D.~Manchon's article \cite{Man12}. Both consider simple graphs, so one does not allow loops (edges with a single vertex) or multiple edges. Letting $G \sus \Hy$ be the subspace of graphs, i.e., hypergraphs where all edges have cardinality $\leq 2$, the simple graphs are a quotient of the subalgebra $G$. This goes well, as the coproduct $\delta$ descends, since if a graph $(E,V)$ has a loop, or a multiple edge, this will also be the case with every term in the resulting coproduct by $\delta$. The same is the case for hypergraphs. So one may consider hypergraphs with no loops (edges with only a single vertex), and with no multiple edges, and one will still get bialgebras for $\delta$ and $\Delta$.
\end{remark}

\begin{remark} L.~Foissy also considers the extraction-contraction bialgebra for
  hypergraphs in the arXiv preprint \cite{Fo-Hyp}, appearing
  some months after the present article appeared on the arXiv.
  He works in the setting of species. The associated Hopf algebra
  of hypergraphs has a coproduct which is a variation of ours. Rather
  than summing over subsets of edges as in \eqref{eq:extcon-delta}, he
  sums over surjections $V \ppil W$ such that restrictions
  to the fibers are connected hypergraphs.
  This hypergraph bialgebra of \cite{Fo-Hyp} is essentially a quotient
  of ours. Quotienting out by:
  \begin{itemize}
 \item Hypergraphs with an edge containing a single vertex
 \item The differences between a hypergraph and its reduced version by
   removing multiple edges (but keeping one edge)
  \end{itemize}
  this gives a quotient bialgebra,
  as is readily verified. In this bialgebra, for any hypergraph with a connected
  component consisting of a single
  vertex, this component must be an isolated vertex (contained in no edges).
  The bialgebra of \cite{Fo-Hyp} is obtained by modifying
  every hypergraph in the resulting bialgebra
 by:  \begin{itemize}
      \item adding an empty edge, and
      \item   replace every one vertex  isolated component $(\emptyset, v)$ with
  $(\{v\},v)$,
\end{itemize}
  \end{remark}


\subsection{Coproduct $\delta^d$}

Let $H^d \sus H$ be the subspace of hypergraphs with no isolated vertices,
i.e. every vertex is contained in at least one edge. Alternatively
$V = V^*$ or equivalently $V_{|\emptyset} = \emptyset$.
So $H^d \sus H$ is a sub-bialgebra of $(H, \mu, \Delta)$
as the coproduct $\Delta^d$ restricts to $H^d$. 

Since $d(H^\circ) = H^d$, from 
Subsection \ref{subsec:EC-d1} we get a
comodule coproduct:
\begin{equation*} \delta^d : H \pil H^d \te H, \quad 
  \delta^d =  (d \te d) \circ \delta \circ d.
\end{equation*}
We describe this explicitly. 
Consider discrete hypergraphs $(G,G)$ given by the canonical $G \pil P(G)$
and commutative diagrams from maps $\phi : E \pil G$:
\begin{equation} 
  \xymatrix{ V \ar[r]^{h^d} \ar[d]_{\psi} & P(E) \ar[d]^{P\phi} \\
    G \ar[r] & P(G)}.
 \end{equation}

\begin{definition} 
Let $V^* \sus V$ be the vertices which are incident to at least one edge. The hypergraph $(V,E)$ has {\it connected edge set} if the only such diagram for $(V^*,E)$  with $E \ppil G$ surjective is when  $G$ is a single point.
This means the equivalence relation on $E$ generated by $e \sim e^\prime$ if $e$ and $e^\prime$ share a common vertex, only has one equivalence class, the whole of $E$. 
\end{definition}

Note that if each edge is incident to a vertex and each vertex is
incident to an edge, the two notions of connectedness coincide.
But if there are vertices not on any edge, the graph may have connected edge
set but will not have connected vertex set.
Similarly if there are edges with an empty vertex set. 

Given $(E,V)$ let $(E_k,V^*_k), k \in K$ be the
connected components for $(E,V^*)$ in the above
relation. Let $E^\prime = \cup_{k \in K}  E_k$, and $G = K \cup (E \backslash
E^\prime)$ (the edges of $E \backslash E^\prime$ are edges with no vertices).
We have a surjection $E \ppil G$, and a commutative diagram
for $(E,V^*)$ and $(G,G)$ which is initial among such diagrams.

For each subset $U \sus V^*$, we get a sub-hypergraph $(E,U)$ (here each edge $e$ is contracted onto $U$, i.e. we remove from each edge the vertices in $V \backslash U$), with connected components $(F_\ell,U_\ell), \ell \in L$. Let  $F  = \cup_{\ell \in L} F_\ell$ and $E/U = L  \cup (E\backslash F)$.  This gives a surjection $E \ppil E/U$ where for each vertex $u \in U$ all edges containing $u$ are mapped to a single element in $E/U$. So we have a commutative diagram
with  $(E,U)$ and  $(E/U,E/U)$, which is initial among all commutative diagrams to discrete hypergraphs.

The comodule coproduct is then:
\[ \delta^d : H \pil H^d \te H, \quad
  (E,V) \overset{\delta^d}{\longmapsto} \sum_{U \sus V^*} (E,U) \te (E/U,U^c). 
\]
The comodule coproduct $\delta^d$ restricts to a coproduct
$\delta^d : H^d\pil H^d \te H^d$. (By abuse of notation we also use
$\delta^d$ for its restriction to $H^d$.)
There is a counit $\vep^d_\delta$:
\[ (E,V) \overset{}{\mapsto} \begin{cases} 1, & V = \emptyset \\
    0, & \text{ otherwise } \end{cases}. 
\]
which is the same as the map
$\epsilon_\Delta$ restricted to $H^d$.
This gives a bialgebra $(H^d,\mu, \delta^d, \eta, \vep^d_\delta)$.

\begin{proposition}
  The bialgebra $(H, \mu, \Delta^d, \eta, \epsilon^d_\Delta)$
  is a comodule bialgebra
  over the bialgebra $(H^d,\mu, \delta^d, \eta, \vep^d_\delta)$.

As consequence we have a double bialgebra $(H^d, \mu, \Delta^d, \delta^d)$.
\end{proposition}

\begin{proof} 
This is similar to the proof of Proposition \ref{thm:EC-cointer1}.
\end{proof}


\subsection{Coproduct $\delta^c$}

Let $H^c \sus H$ be the image of $H^\circ \sus H$ by complementation $c$.
Let $E^\times \sus E$ denote the edges whose vertex sets
are not the whole of $V$.
Then $H^c$ are generated by the hypergraphs $(E,V)$ where no edge contains all
the vertices. Alternatively hypergraphs with 
$E = E^\times $ or $\lk_EV  = \emptyset$.
Again $H^c \sus H$ is a sub-bialgebra of $(H, \mu^c, \Delta^c)$. 
Using the complement involution we get a coproduct:
\begin{equation*}    \delta^{c} : H \pil H^{c} \te H, \quad
	\delta^c =  (c \te c) \circ \delta \circ c. 
\end{equation*}
For a hypergraph $(F,V)$ denote the complement hypergraph as $(\ov{F},V)$.
So the latter is given by the composition $F \mto{h} P(V) \mto{c} P(V)$.
For $F \sus E^\times$ consider surjections $V \ppil W$ such that for each
$f \in F$, the complement vertex set $h(f)^c \sus V$ maps to a single point
in $W$. Let $V \ppil V/\ov{F}$ be initial among such surjections.
The comodule coproduct $\delta^c$ is then:
\[  \delta^{c} : H \pil H^{c} \te H, \quad
	(E,V) \overset{\delta^c}{\mapsto} \sum_{F \sus E^\times} (F,V) \te (F^c,V/\ov{F}). 
      \]
      This again restricts to a coproduct on $H^c$ giving a double bialgebra
      $(H^c, \mu^c, \Delta^c, \delta^c)$. 


\subsection{Coproduct $\delta^{cd}$}

Let $H^{cd} \sus H$ be the image of $H^\circ \sus H$ by $c \circ d$.
Let $V^\times \sus V$ denote the vertices not on every edge of $E$.
Then $H^{cd}$ are generated by the hypergraphs $(E,V)$
where no vertex is contained
in all edges. Alternatively hypergraphs with 
$V = V^\times $ or $\co_VE  = \emptyset$.
Again $H^{cd} \sus H$ is a sub-bialgebra of $(H, \mu^c, \Delta^{cd})$. 
Using the complement dual involution we get a coproduct:
\begin{equation*}    \delta^{cd} : H \pil H^{cd} \te H, \quad
      \delta^{cd} = (c\circ d \te c \circ d) \circ \delta \circ (c \circ d). 
\end{equation*}
For $U  \sus V^\times$  consider surjections $E \ppil G$ such that
for each $u \in U$, all edges not containing $u$  map to a single point in $G$.
Let $E \ppil E/\ov{U}$ be initial among such surjections.

The comodule coproduct is then:
\begin{equation*}  \delta^{cd} : H \pil H^{cd} \te H, \quad
  (E,V) \overset{\delta^{cd}}{\mapsto} \sum_{U \sus V^\times} (E,U) \te
  (E/\ov{U}, U^c). 
\end{equation*}
Note that vertices in the intersection of all edges,
only occur on the right side of the tensor product.
The map $\delta^{cd}$ again restricts to a coproduct on $H^{cd}$ giving
a double bialgebra
      $(H^{cd}, \mu^c, \Delta^{cd}, \delta^{cd})$. 


\subsection{Examples with $\delta$-coproducts}
\label{subsec:coint-ex}

Here are the various comodule coproducts for the path graph with three vertices:
\begin{align*}
  \cherry  & \overset{\delta} \longmapsto
             \trev \te \cherry  + 2 \, \linvert \, \point \te \linvert
             + \cherry \te \point
            \\
  \cherry & \overset{\delta^d} \longmapsto \edge \, \edge \te \cherry +
            2 \, \pointl \, \edge \te \cher + \pointll \te \linvert
            + 2 \cher \te \pointl + \pointl \pointl \te \pointl
            + \cherry \te \edge \\
   \cherry & \overset{\delta^{c}} \longmapsto \trev \te \cherry +
             2 \, \point \, \linvert \te \linvert \, \point
             + \cherry \te \trev \\
  \cherry  & \overset{\delta^{cd}} \longmapsto
             \edge \, \edge \edge \te \cherry  + 2 \, \edge \pointl \te \cher
             + \pointl \, \pointl \te \pointll           
\end{align*}

\begin{remark} The path graph is in $H^\circ, H^d, H^c$ but not in $H^{cd}$. 
 The three first coproducts may be taken in
  the extraction-contraction bialgebras, but
  the last case only occurs as a comodule coproduct. 
\end{remark}

Here are the comodule coproducts for the triangle hypergraph
with one edge on three vertices:
  \begin{align*}
    \trekant & \overset{\delta}{\longmapsto} \trev \te \trekant
               + \trekant \te \point \\
   \trekant & \overset{\delta^d}{\longmapsto} \edge \te \trekant
              + 3 \pointl \te \linvert + 3 \linvert \te \pointl
              + \trekant \te \edge \\
    \trekant & \overset{\delta^c}{\longmapsto} \trev \te \trekant
               \\
    \trekant & \overset{\delta^{cd}}{\longmapsto} \edge \te \trekant
  \end{align*}
  \begin{remark} The triangle is in $H^\circ$ and $H^d$, so the first
    two are also coproducts in the extraction-contraction algebras.
    The last two cases only occurs as comodule coproducts.
\end{remark}


\section{Associated polynomials}
\label{sec:poledge}

In the following $\kk$ is a field of characteristic zero, as we relate
to \cite{Fo22}.
The polynomial algebra $\kk[x]$ has a Hopf algebra
structure $(\mu, \Delta, 1, \epsilon_\Delta)$ where
$$
	\Delta(x) = x \te 1 + 1 \te x,
$$ 
and a bialgebra structure $(\mu, \delta, 1, \epsilon_\delta)$ where 
$$
	\delta(x) = x \te x.
$$
We shall identify $\kk[x] \te \kk[x]$ as $\kk[x,y]$, so these
maps can be written
\[ 
	x \mapsto x+y \text{ resp.~} x \mapsto x\cdot y. 
\]
If  $(B, \mu, \Delta, \delta)$ is a connected double bialgebra, by
the main result of \cite{Fo22} there is a unique algebra morphism
$\Phi : B \pil \kk[x]$ such that $\Phi$ is a bialgebra morphism
$(B,\mu,\Delta) \pil (\kk[x],\mu,\Delta)$ and a bialgebra
morphism $(B,\mu,\delta) \pil (\kk[x], \mu, \delta)$.
We compute the values of this for some simple hypergraphs.


We apply this to the double bialgebra $(H^\circ, \mu, \Delta, \delta)$.
Note that $(H^\circ,\mu,\Delta)$ is 
a connected bialgebra.
Hence there is a unique double bialgebra morphism
\begin{equation} \label{eq:pol-chio}
  \chi^\circ : (\Hy^{\circ},\mu, \Delta, \delta) \pil
  (\kk[x], \mu, \Delta, \delta). 
\end{equation}
We first compute this for some simple hypergraphs.
In the following let $n$ be the cardinality $|V|$ of the vertex set $V$.

\begin{proposition} \label{pro:ass-noedge}
  Consider the hypergraph $(\emptyset, V)$ with no edges.
  Then 
$$
	\chi^\circ(\emptyset, V) = x^n.
$$
\end{proposition}

\begin{proof}
  Due to $(\emptyset, V) = \prod_{v \in V} (\emptyset, \{v\})$, it
  is enough to show that $p(x) := \chi(\emptyset, \{v \}) = x$.
  We have:
\[ 
  	(\emptyset, \{v\}) \overset{\Delta}{\longmapsto}
    (\emptyset, \emptyset) \te (\emptyset, \{v\})   + 
    (\emptyset, \{v \}) \te (\emptyset, \emptyset). 
\]
Mapping this to $\kk[x]$ and $\kk[x] \te \kk[x] \iso \kk[x,y]$
we get the equation
    \[ 
    	p(x+y) = 1 \cdot p(x) + p(y) \cdot 1. 
\]
    We also have:
 \[ 
 	(\emptyset, \{v\}) \overset{\delta}{\longmapsto}
    (\emptyset, \{v\}) \te (\emptyset, \{v\}), 
\]   
giving $p(xy) = p(x) \cdot p(y)$. These two equations give $p(x) = x$. 
\end{proof}

\begin{proposition}
  Let $(e,V)$ be the hypergraph consisting of a single edge $e$ containing all
  the vertices of $V$. Then for $n \geq 1$:
 $$
 \chi^\circ(e,V) = x^n - x.
$$
\end{proposition}

\begin{remark} We see that for the hypergraph with a single edge containing
  a single vertex $(e, \{v \})$,
  the associated polynomial is zero, $ \chi(e,\{v \})=0$.
  However, by Proposition \ref{pro:ass-noedge}, if we have a single vertex
  and no edge $(\emptyset, \{v\})$, the associated polynomial is $x$.
\end{remark}

\begin{proof} Let $p(x) = \chi(e,V)$. 
 \[ 
 	(e,V) \overset{\delta}{\longmapsto}
    (\emptyset, V) \te (e,V) + (e, V) \te (\emptyset, \{*\}). 
\]   
Again mapping this to $\kk[x]$ and $\kk[x] \te \kk[x] \iso \kk[x,y]$ we get the
equation
\[ 
	p(xy) = x^n \cdot p(y) + p(x)  \cdot y. \]
The only solutions are $p(x) = a(x^n - x)$. 
The $\Delta$-coproduct is:
  \[ 
  	(e, V) \overset{\Delta}{\longmapsto} \een \te (e,V)
     + \sum_{\emptyset \subset U \subset V} (\emptyset, U) \te (\emptyset, U^c) + 
   (e,V) \te \een . 
 \]
  Mapping to $\kk[x]$ and $\kk[x,y]$, this gives for instance for $n = 3$:
  \[ 
  p(x+y) = p(x) + 3x^2y + 3xy^2 + p(y).
 \]
We easily verify  that $a = 1$.
\end{proof}

By sending hypergraphs $(E,\emptyset) \mapsto 1$, \eqref{eq:pol-chio}
may be extended to a bialgebra morphism
\begin{equation} \label{eq:pol-chi}
  \chi : (\Hy,\mu, \Delta) \pil (\kk[x], \mu, \Delta).
\end{equation}


\section{The chromatic polynomial for hypergraphs}
\label{sec:chrompol}

The map $\chi$ in the previous section gives for
each hypergraph $(E,V)$ a polynomial $\chi_{E,V}(x)$. For classical
graphs, $\chi_{E,V}(n)$ is the chromatic polynomial counting
the number of colorings of the vertices
of $(E,V)$ with
$n$ colors, such that on every edge the two
vertices have different colors, see \cite{FoChrom}.

For general hypergraphs we now show
that it counts the number of colorings of vertices, using $n$ colors,
such that no edge is monochromatic. So in this case, perhaps more informative
than {\it chromatic} polynomial would be to call it the
{\it non-monochromatic} polynomial. However the former is standard
in the literature.


\subsection{Colorings}

By \eqref{eq:extcon-delta}
for the bialgebra coproduct $\delta^\circ$, the image of $(E,V)  \in H^\circ$
is a sum of terms for $F \sus E$: 
\[ 
	(F,V) \te (F^c, V/F). 
\]

\begin{lemma} Let $U \sus V$ be the vertex support of edges in $(F,V)$
  (i.e.~all vertices incident to some edge in $F$). 
  The hypergraph $(F^c,V/F)$ has no loops (edges with a single vertex) if and only if
  for each connected component $(F_i,U_i)$ of $(F,U)$, the
  edge set $F_i$ is induced, i.e.,~$F_i = F_{|U_i}$.
\end{lemma}

\begin{proof}
  If there was an edge in $F_{|U_i}$ which is not in $F_i$, then this edge would
  become a loop in $(F^c, V/F)$ as the vertices of $U_i$ map to a
  single point in $W$.
\end{proof}

\begin{definition}
  A {\it coloring} of the hypergraph $(E,V)$ by a set $C$, whose
  elements are called colors, 
  is a map $V \pil C$ such that no edge $e \in E$ is monochromatic,
  i.e., for each edge not all the vertices have the same color.
\end{definition}

In particular note:
\begin{itemize}
  \item If there is some edge with only one vertex,
    the hypergraph $(E,V)$ has no coloring,
  \item If $V \neq \emptyset$ and $C = \emptyset$, there is no coloring
    since there is no map $V \pil C$.
  \item If $V = \emptyset$ there is a unique map $V \pil C$ for
    every $C$ so there is {\it one} coloring for each $C$.
  \end{itemize}
Let $\chi_{E,V}(n)$ be the number of colorings of a hypergraph $(E,V)$ with
$n$ colors. We see that $\chi_{E,V}(0) = 0$ when $|V| \geq 1$.
From \cite{Doh,Hel} it follows that $\chi_{E,V}(n)$ is a polynomial,
the {\it chromatic polynomial} of the hypergraph $(E,V)$.

\begin{example} {} \hfill

\begin{itemize}
    \item When $V = \emptyset $ then
$\chi_{E,V}(n) = 1$ for all $n \geq 0$.
\item When  $E = \emptyset$, then $\chi_{E,V}(n) = n^{|V|}$. 
\item When $E$ is a single edge containing all the vertices $V$,
  then $\chi_{E,V}(n) = n^{|V|} - n$, the polynomial
  we computed in the previous section.
\end{itemize}
\end{example}

\begin{theorem} \label{thm:col-chi} Let $\kk = \QQ$.
  The double bialgebra morphism $\chi^\circ$ of
  \eqref{eq:pol-chio}
  maps the hypergraph $(E,V)$ to the chromatic polynomial $\chi_{E,V}(x)$.
\end{theorem}

We prove this at the end of this section.

  \begin{remark}
     Colorings of hypergraphs
     are discussed in \cite[Chap.19]{Ber} but concerns the chromatic number,
     and not the polynomial.
In particular, this is defined by requiring the edges not
to be monochromatic. This definition of the chromatic number of
a hypergraph was introduced by Erd\"os and Hanal \cite{EH}.
The chromatic {\it polynomial} for hypergraphs seems first to have
been discussed
by T.~Helgason in \cite{Hel}. Somewhat later K.~Dohmen \cite{Doh}
discusses basic properties
of this polynomial (based on his Ph.D.~dissertation) and shows among other things that its
coefficients are integers.
I.~Tomescu \cite{Tom} computes the polynomial for classes of hypergraphs,
and consider its coefficients. R.~Zhang and F.~Dong \cite{ZD} is
a recent presentation of various facts
and properties concerning this polynomial, as well as some
open questions, especially concerning its zeroes.
Algorithmic aspects concerning colorings of uniform hypergraphs
are considered in \cite{Su,Di}, where one is concerned with finding a
$c$-coloring when $c \geq k$ and the hypergraph is promised to be $k$-chromatic.

In \cite[Sec.16,  18.1]{AA} M.~Aguiar and F.~Ardila give a general setting
associating to characters on Hopf monoids a polynomial invariant.
For the character
$\epsilon_\delta$ of the (restriction) bialgebra of graphs
this becomes the chromatic polynomial.
The extension of this to hypergraphs using
the restriction bialgebra and character $\epsilon_\delta$,
gives the hypergraph polynomial we consider here.

Using instead the restriction-descent bialgebra for hypergraphs
considered in Subsection \ref{subsec:res-descent} (originally
in \cite{AA}), 
Aval, Karabaghossian, and Tanasa \cite{AKT} extend (essentially) the
character $\epsilon_\delta$,
and get another hypergraph polynomial, see Remark \ref{rem:res-contract}.
\end{remark}


\subsection{Colorings built from two disjoint color sets}

Before proving Theorem \ref{thm:col-chi} we develop some auxiliary results.

\begin{proposition} Let $C = X \cup Y$ be a disjoint union. The colorings of $(E,V)$ by $C$ is in bijection with the following data set:
  \begin{itemize}
    \item A subset $U \sus V$, 
    \item A coloring of $(E_{|U}, U)$ by $X$,
    \item A coloring of $(E_{|U^c},U^c)$ by $Y$.
    \end{itemize}
\end{proposition}

\begin{proof}
If the map $V \pil C$ gives a coloring, we get a subset $U \sus V$, those vertices colored
by $X$ and the complement $U^c$ colored by $Y$.
The map $U \pil X$ then is a coloring of $(E_{|U}, U)$ and
$U^c \pil Y$ is a coloring of $(E_{|U^c}, U^c)$.
Conversely, given such a data set, we get a map $V \pil X \cup Y$.
This gives a coloring, since if $e \in E$ is an edge not in $E_{|U}$
or in $E_{|U^c}$, then $e$ contains vertices from both
$U$ and $U^c$ and these will have different colors.
\end{proof}

\begin{corollary} \label{cor:colXuY}
The number of colorings $\chi_{E,V}(n)$ is a polynomial in $n$ with zero constant term if $V$ is not empty, and 
  \begin{equation} \label{eq:col-nm}
    \chi_{E,V}(n+m) = \sum_{U \sus V}
    \chi_{E_{|U},U}(n) \cdot \chi_{E_{|U^c}, U^c}(m).
    \end{equation}
\end{corollary}

\begin{proof}
  The identity above is immediate from the proposition above.
  
  We show that $\chi_{E,V}(n)$ is a polynomial by induction on
  the cardinality of $E$ and of $V$, with zero constant term
  when $V \neq \emptyset$.

  If $V$ is empty $\chi_{E,V}(n)$ is the constant polynomial $1$.
  If $E$ is empty,
  $\chi_{E,V}(n) = n^{|V|}$.
Let then $m = 1$, the difference $\chi_{E,V}(n+1)- \chi_{E,V}(n)$ is
a polynomial for $n \geq 0$, by induction and the expression
\eqref{eq:col-nm} above.
  Whence $\chi_{E,V}(n)$ is a polynomial for $n \geq 1$. But the
  above identity gives by induction when $n = m$ that
  the constant term of $\chi_{E,V}(2n)$ is twice the constant
  term of $\chi_{E,V}(n)$, and so this constant term must be zero.
\end{proof}

Now we consider colorings by $X \times Y$. 

\begin{proposition}
  Let $C = X \times Y$. The colorings of $(E,V) \in H^\circ$ by $C$ is in bijection
  with the following data set:
  \begin{itemize}
  \item A subset $F \sus E$ and a coloring of $(F,V)$ by $X$,
  \item A coloring of $(F^c, V/F)$ by $Y$.
  \end{itemize}
  Note that in any such coloring, both hypergraphs $(F,V)$ and
  $(F^c, V/F)$ must be loopless.
\end{proposition}

\begin{proof}
  Given a coloring by $X \times Y$, let $F \sus E$ be all edges
  which are monochromatic for the composition $V \pil X \times Y \pil Y$.
  Let $U$ be the vertex support of $F$, and let $I$ index the
  connected components of $(F,U)$. Then $V/F = I \cup (V \backslash U)$
  and the map $V \pil Y$ descends to a map $V/F \pil Y$.
  Clearly the composition $V \pil X \times Y \pil X$ gives
  a coloring of $(F,V)$. Furthermore the composition $V \pil Y$
  gives a coloring of $(F^c, V/F)$, since the edges in $F^c$ are
  non-monochromatic for $Y$.

  Conversely, given a data set as stated in the proposition above,
  we get a map $V \pil X \times Y$,
  and this is obviously a coloring of $(E,V)$. Furthermore these
  correspondences are seen to be inverse to each other. 
\end{proof}

\begin{corollary} \label{cor:colXxY}
  The number of colorings fulfils the identity
  \[
  	\chi_{E,V}(nm) = \sum_{F \sus E}
      \chi_{F,V}(n) \cdot \chi_{F^c,V/F}(m), 
 \]
 where $(F,V)$ and $(F^c,V/F)$ are as in the definition
 (of the restriction $\delta^\circ$) of $\delta$, in \eqref{eq:extcon-delta}. 
\end{corollary}

\begin{proof}
  This follows by the proposition above. Note that if $(F,V)$ has a loop
  then $\chi_{F,V}(n)$ is zero and if $(F^c,V/F)$ has a loop, then
  $\chi_{F^c,W}(m)$ is zero.
\end{proof}

\begin{proof}[Proof of Theorem \ref{thm:col-chi}.]
This follows from Corollaries \ref{cor:colXuY} and
\ref{cor:colXxY}.
\end{proof}

\begin{remark}
As mentioned in \eqref{eq:pol-chi} $\chi^\circ$ extends
to a bialgebra homomorphism $\chi : (H,\mu,\Delta) \pil (\kk[x], \mu, \Delta)$,
sending each $(E,\emptyset) \mapsto 1$.
This will be convenient as we then to any hypergraph have a chromatic
polynomial. However for the $(\{e\}, V)$ the single edge containing
all the vertices $V$, the hypergraph polynomial is $x^n - x$ where
$n = |V|$. This formula is valid for $n \geq 1$, but not for $n = 0$,
when the chromatic polynomial by our extension is $1$. This suggests
the naturality of the double bialgebra setting $(H^\circ, \Delta, \delta)$
for the chromatic polynomial, and that the extension has some lack
of naturality.
\end{remark}


\section{Four chromatic polynomials}
\label{sec:quartpol}

Associated to the double bialgebra $(H^d,\mu,\Delta^d, \delta^d)$, there
is also as above a unique double bialgebra morphism to
$(\kk[x], \mu, \Delta, \delta)$. It gives a polynomial $\chi^d_{E,V}(x)$, and
as in the end of the previous section, we may extend its domain to
all hypergraphs $(E,V) \in H$. Then
$\chi^d_{E,V}(n)$ counts the number of colorings of the edges of
$(E,V)$, with $n$ colors,
such that there is no vertex where all the edges incident to
this vertex have the same color.

In particular for a tree, this polynomial would always vanish, as
every tree has a leaf. However one may remove the vertex leaves, to get
trees where the ``leaves'' are edges. This gives non-trivial polynomials.

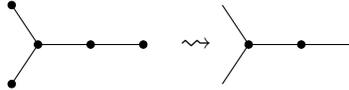
\begin{center}
\begin{tikzpicture}[scale = 0.7]
\draw[] (0,0)--(-0.5,0.75) node[anchor=west] at (-0.33,0.5){};
\draw[] (0,0)--(-0.5,-0.75) node[anchor=west] at (-0.33,-0.5)  {};
\draw[] (0,0)--(1,0) node[anchor=north] at (0.5,0)  {};
\draw[] (1,0)--(2,0) node[anchor=north] at (1.5,0)  {};
\filldraw (0,0) circle (2pt);
\filldraw (-0.5,0.75) circle (2pt);
\filldraw (-0.5,-0.75) circle (2pt);
\filldraw (1,0) circle (2pt);
\filldraw (2,0) circle (2pt);

\draw node at (3,0) {$\rightsquigarrow$};

\draw[] (4,0)--(3.5,0.75) node[anchor=west] at (3.67,0.5){};
\draw[] (4,0)--(3.5,-0.75) node[anchor=west] at (3.67,-0.5)  {};
\draw[] (4,0)--(5,0) node[anchor=north] at (4.5,0)  {};
\draw[] (5,0)--(6,0) node[anchor=north] at (5.5,0)  {};
\filldraw (4,0) circle (2pt);
\filldraw (5,0) circle (2pt);

\end{tikzpicture}
 \end{center}

\medskip
From the four double bialgebras
\[ 
	H^\circ, \quad H^d, \quad H^c, \quad H^{cd} 
\]
we get four chromatic polynomials associated to every hypergraph
$(E,V)$ in $H$:
\[ 
	\chi_{E,V}(x), \quad \chi^d_{E,V}(x), \quad \chi^c_{E,V}(x), \quad
  \chi^{cd}_{E,V}(x) . 
\]





\subsection{Examples}

\begin{example}
  Consider the hypergraph $h$ with $a$ vertices and $m$ edges, each edge
  containing all the vertices, alternatively phrased: We have an
  $m$-tuple edge containing all vertices. We
  represent this below to the left, together with its derived hypergraphs:
  \[ 
  	h = \pointli{a}{m}, \quad h^d = \pointli{m}{a}, \quad
    h^c = \pointe{a}{m},
    \quad h^{cd} = \pointe{m}{a}. 
 \]
  In the last two hypergraphs the edges are empty  and vertices are
  non-incident. Note that (for $a,m \geq 1$)
  $h$ is in $H^\circ$ and $H^d$ but not in $H^c$ nor $H^{cd}$. 
  As $0,1$-matrices, $h$ is the $m\times a$-matrix
  with all entries $1$, and $h^d$ its transpose. Further $h^c$ is the
  $m \times a$-matrix with all entries $0$ and $h^{cd}$ its transpose.
  Their chromatic polynomials are (for $a,m \geq 1 $):
 \[ 
 	\chi(x) = x^a - x, \quad \chi^d(x)= x^m - x, \quad \chi^c(x) = x^a,
    \quad \chi^{cd}(x) = x^m. 
 \]
\end{example}

\begin{example}
Consider the hypergraph $h$ given schematically as:
\[ 
  	h = \trekantl{a}{b}{c}{\ell}{m}{n}. 
\]
  By this we mean vertex sets $A,B,C$ consisting of $a,b$ and $c$ vertices,
  an $\ell$-multiple edge consisting of $A \cup B$, an $m$-multiple edge
  consisting of $A \cup C$ and an $n$-multiple edge consisting of
  $B \cup C$. Its derived hypergraphs are:
  \[ 
  	h^c = \pointli{c}{\ell} \, \pointli{b}{m} \, \pointli{a}{n}, \quad
    	h^d = \trekantl{\ell}{m}{n}{a}{b}{c}, \quad
    	h^{cd} = \pointli{\ell}{c} \, \pointli{m}{b} \, \pointli{n}{a}
      \]
      If all integers here are positive, these are in all four double
  bialgebras and the 
  chromatic polynomials are:
  \begin{align*} 
  \chi(x) 	&= x^{a+b+c} - x^{a+1} - x^{b+1} - x^{c+1} + 2x\\
  \chi^c(x) &= (x^a-x)(x^b-x)(x^c-x) \\
    \chi^d(x) &=  x^{\ell + m + n} - x^{\ell+1} - x^{m+1} - x^{n+1} + 2x\\
   \chi^{cd}(x) &= (x^\ell-x)(x^m-x)(x^n-x).
      \end{align*}
\end{example}

\begin{example} 
All trees with $n$ edges have the same chromatic
  polynomial $x(x-1)^n$. However the four chromatic polynomials
  refine things considerably. There are two trees with three edges,
  the path and the star:
 \begin{center}
\begin{tikzpicture}[scale = 0.7]
\draw[] (0,0)--(1,0) node[anchor=west] at (-0.33,0.5){};
\draw[] (1,0)--(2,0) node[anchor=west] at (-0.33,-0.5)  {};
\draw[] (2,0)--(3,0) node[anchor=north] at (0.5,0)  {};
\filldraw (0,0) circle (2pt);
\filldraw (1,0) circle (2pt);
\filldraw (2,0) circle (2pt);
\filldraw (3,0) circle (2pt);

\draw[] (6,0)--(7,0) node[anchor=west] at (-0.33,0.5){};
\draw[] (7,0)--(7.5,0.85) node[anchor=west] at (-0.33,-0.5)  {};
\draw[] (7,0)--(7.5,-0.85) node[anchor=north] at (0.5,0)  {};
\filldraw (6,0) circle (2pt);
\filldraw (7,0) circle (2pt);
\filldraw (7.5,0.85) circle (2pt);
\filldraw (7.5,-0.85) circle (2pt);
\draw node at (7.7,0) {.};
\end{tikzpicture}

\end{center}

The chromatic polynomials of the path are:
\begin{equation*} \chi(x) = x(x-1)^3,
    \quad  \chi^d(x)= 0, \quad 
    \chi^c(x) =  x(x-1)^3,
    \quad  \chi^{cd}(x) = 0.
\end{equation*}

The chromatic polynomials of the star are:
\begin{equation*} \chi(x) = x(x-1)^3,
    \quad  \chi^d(x)= 0, \quad 
    \chi^c(x) =  x^2(x-1)(x-2),
    \quad \chi^{cd}(x) = x(x-1)(x-2).
  \end{equation*}
\end{example}


\subsection{When all chromatic polynomials are zero}

It is not hard to come up with examples where all four chromatic
polynomials vanish. The chromatic polynomial $\chi_h(x)$ of a hypergraph
$h$ vanishes precisely when it has an edge with exactly one vertex.
Using the representation of the hypergraph as $0,1$-matrix $h$,
we then have:
\begin{itemize}
  \item $\chi_h(x) = 0$ iff $h$ has a row with exactly one $1$.
  \item $\chi_h^d(x) = 0$ iff $h$ has a column with exactly one $1$.
  \item $\chi_h^c(x) = 0$ iff $h$ has a row with exactly one $0$.
  \item $\chi_h^{cd}(x) = 0$ iff $h$ has a column with exactly one $0$.
\end{itemize}

Using this any hypergraph may be extended, adding two vertices and two edges, to a hypergraph where all chromatic polynomials vanish. One needs only add rows and columns with the above properties to the $0,1$-matrix.

 
\section{Further bialgebras}
\label{sec:newq}

We give three more quartets of bialgebras of hypergraphs: The first consisting
of restriction-descent bialgebras, the second of descent
bialgebras, and the last of extraction-contraction bialgebras.


\subsection{Restriction-descent bialgebras}
\label{subsec:res-descent}
Ardila and Aguiar in \cite[Subsec.3.1]{AA} introduce a variation on
Schmitt's Hopf algebra
of graphs \cite{Schmitt}. This generalizes readily to hypergraphs
\cite[Subsec.20.1]{AA}.
Aguiar and Ardila essentially do not allow empty edges. In the graph
case they do not have empty edges. In the hypergraph case they require
the hypergraph to have exactly one empty edge. But these are of course
in bijection with hypergraphs with no empty edge.

We give a slightly more general version of this in that we do not have
any restriction on the edges. 

For a subset $U$ of $V$, let $U^c$ be the complement.
We get a hypergraph $(E,U^c)$ by the composite
$E \pil P(V) \pil P(U^c)$. 
We get a bialgebra of hypergraphs $(H,\mu,\Delta^\prime, \eta, \eps)$
where the coproduct is:
\[
  (E,V) \overset{\Delta^\prime} \longmapsto
  \sum_{U \sus V} (E_{|U}, U) \te (E, U^c). 
\]
For the edge this coproduct would
now be:
\begin{equation} \label{eq:pro-cop3}
  \linvert \, \overset{\Delta^\prime}\longmapsto \, \een \te \linvert +
  2 \bullet \te \pointl + \linvert \te \edge \, ,
\end{equation}
while in $(\Hy, \mu, \Delta)$ the coproduct is:
\[ 
	\linvert \, \overset{\Delta}\longmapsto \, \een \te \linvert +
  2 \bullet \te \bullet + \linvert \te \een.
\]
The counit $\eps$ is as before:
\[ 
	(E,V) \overset{}{\mapsto} \begin{cases} 1, & V = \emptyset \\
    0, & \text{ otherwise } \end{cases}. 
\]
As in Section \ref{sec:hyp} we get four bialgebras:
\[ (\Hy, \mu, \eta, \Delta^\prime, \eps), \quad
  (\Hy, \mu, \eta, \Delta^{\prime d}, \eps^{d}), \quad
  (\Hy, \mu^c, \eta, \Delta^{\prime c}, \vep^{c}), \quad
  (\Hy, \mu^c, \eta, \Delta^{\prime cd}, \vep^{cd}). 
\]

\medskip
By sending a hypergraph $(E,V) \mapsto (E^*,V)$ (i.e. omit all empty edges),
we get a connected {\it quotient} bialgebra
$(H^\circ, \mu, \Delta^\prime)$ now sending the edge to:
\begin{equation*} 
  \linvert \, \overset{\Delta^\prime}\longmapsto \, \een \te \linvert +
  2 \bullet \te \pointl + \linvert \te \een.
\end{equation*}
This is essentially the Hopf algebra of hypergraphs in \cite[Subsec.~20.1]{AA}.

Let $\He$ be the hypergraphs with empty vertex set and coproduct:
\[ (E,\emptyset) \overset{\Delta^\emptyset}{\mapsto} (E,\emptyset) \te
    (E,\emptyset). \]
  Tensoring the bialgebras on $H^\circ$ and $\He$ we get
  another bialgebra structure on $H = H^\circ \te H^\emptyset$.
  However we consider
  $(H, \mu, \Delta^\prime, \eta, \eps)$ to be the fundamental one
  as the others are derived from it.

  \begin{remark} \label{rem:res-contract}
    To a Hopf monoid $\bH$ in vector species and a character $\zeta : \bH \pil
    \kk$, and $h \in \bH[I]$,
    where $I$ is a finite set, in \cite[Sec.~16]{AA} 
    is associated a polynomial $\chi_I(h)(x) \in \QQ[x]$. This construction
    is also investigated in detail in \cite{Fo22}. In particular there
    are in many cases \cite[Sec.~17]{AA} a natural {\it basic character} $\zeta$ 
    giving a {\it basic polynomial invariant}. For graphs this becomes
    the chromatic polynomial.

    The Hopf monoid of hypergraphs considered in \cite{AA} corresponds
    to the restriction-contraction Hopf algebra
    $(H^\circ, \mu, \Delta^\prime)$. 
    Aval, Karabaghossian, and Tanasa \cite{AKT} consider this for the basic character, which to
    a hypergraph associates $0$ if the hypergraph has an edge with two
    or more vertices, and associates $1$ if every edge in the hypergraph
    has a single vertex. To a hypergraph $h$ with vertex set $I$ is then associated a polynomial
    $\chi_I^{AKT}(h)$ which is distinct from the hypergraph chromatic polynomial
    $\chi_h$ we consider here. This latter polynomial $\chi_h$
    is associated to the basic
    character of the restriction-restriction bialgebra $(H^\circ, \mu, \Delta)$.

    The character $\chi^{AKT}$ gives (see \cite{Fo22}) a bialgebra morphism
    and $\chi$ gives a {\it double} bialgebra morphism
    \[ \chi^{AKT} : (H^\circ, \mu, \Delta^\prime) \pil (\QQ[x], \mu, \Delta),
      \qquad  \chi : (H^\circ, \mu, \Delta, \delta) \pil
    (\QQ[x], \mu, \Delta, \delta). \]
  The first is one of many possible bialgebra morphisms.
  On the other hand the double bialgebra morphism $\chi$ is unique,
  \cite{Fo22}. The character $\chi^{AKT}$ is also considered in
  \cite[Prop.~2.2]{Fo-Hyp}.

  At the end of this section we pose the question of whether
$(H^\circ, \mu, \Delta^\prime)$ has a cointeracting bialgebra.
  \end{remark}

\subsection{Descent-descent bialgebras}  
    \label{subsec:des-des}
One may also get a bialgebra of hypergraphs
$(H,\mu,\Delta^{\prime\prime}, \eta, \eps)$
where the coproduct is:
\[ 
	(E,V) \overset{\Delta^{\prime \prime}} \longmapsto
  \sum_{U \sus V} (E, U) \te (E, U^c). 
\]
For the edge this coproduct would
now be:
\begin{equation*} 
  \linvert \, \overset{\Delta^\prime}\longmapsto \, \edge \te \linvert +
  2 \pointl \te \pointl + \linvert \te \edge \, .
\end{equation*}
The counit $\eps$ is as above.

In this case the coproduct we get by complementing $\Delta^{\prime \prime \,c}$
is the same as $\Delta^{\prime \prime}$. However the products
$\mu$ and $\mu^c$ are distinct, the latter given as in Subsection
\ref{subsec:hyp-products}.  
So we get again four distinct bialgebras:

\[ (\Hy, \mu, \eta, \Delta^{\prime \prime}, \eps), \quad
  (\Hy, \mu, \eta, \Delta^{\prime \prime \, d}, \eps^{d}), \quad
(\Hy, \mu, \eta, \Delta^{\prime \prime \, c}, \eps), \quad
  (\Hy, \mu, \eta, \Delta^{\prime \prime \, cd}, \eps^{d}).  
\]

Again by sending a hypergraph $(E,V) \mapsto (E^*,V)$
(i.e. omit empty edges),
we get a connected {\it quotient} bialgebra
$(H^\circ, \mu, \Delta^{\prime \prime})$ now sending the edge to:
\begin{equation*} 
  \linvert \, \overset{\Delta^\prime}\longmapsto \, \een \te \linvert +
  2 \pointl \te \pointl + \linvert \te \een.
\end{equation*}

\noindent {\it Note:} The above descent-desecent bialgebra was added to the
submission version of this article a few months after the first version
appeared on the arXiv.
At practically the same day there appeared an arXiv preprint \cite{Fo-Hyp}
by L.~Foissy, where he also introduces the above descent algebra.


\subsection{Restriction-contraction bialgebras II}
\label{subsec:res-contract2}
In \cite{Fo-Hyp} mentioned above, L.~Foissy,
shows that the above descent-descent bialgebra has a 
cointeracting bialgebra, an extraction-contraction bialgebra,
\cite[Section 1.3]{Fo-Hyp}. We give this bialgebra
with some slight modifications to our setting.

For a surjection $p : V \ppil W$ and $w \in W$, let $V_w = p^{-1}(w)$ be
the inverse image. Since $V_w \sus V$ we get a restriction map
$P(V) \pil P(V_w)$. A hypergraph $(E,V)$ induces a hypergraph $(E,V_w)$
via the composition $E \pil P(V) \pil P(V_w)$. Let $E_w \sus E$ be the
subset of edges which contain a vertex from $V_w$. We now define
a comodule map $H \overset{\delta^{\prime \prime}}{\pil} H^\circ \te H$ by sending
\begin{equation} \label{eq:newq-cop}
  (E,V) \mapsto \underset{V \ppil W}{\bigoplus}
  (\prod_{w \in W} (E_w,V_w)) \te (E,W),
\end{equation}
where we sum over all sujections $V \ppil W$ such that each $(E_w, V_w)$ is
a connected hypergraph. In particular, an edge $| \mto{\delta^{\prime \prime}}
1 \te |$.

\noindent {\it Note}: The condition of connectedness is only needed for
the product and coproduct to interact well: Then the coproduct is
an algebra morphism.

\begin{example}
We have the following:
\[ \pointll \overset{\delta^{\prime \prime}} \mapsto \pointll \te \pointll, \quad
    \pointll \overset{\Delta^{\prime \prime}} \mapsto
      || \te \pointll + \pointll \te ||. 
  \]
  On the other hand for the cointeracting bialgebras of
  Theorem \ref{thm:EC-cointer1}:
  \[ \pointll \overset{\delta} \mapsto \bullet \te \pointll + 2 \pointl \te
    \pointl + \pointll \te \bullet , \quad
    \pointll \overset{\Delta} \mapsto
      \ben \te \pointll + \pointll \te \ben. 
  \]
\end{example}
  
We define $\vareps^{\prime \prime} : H^\circ \pil \kk$ by sending all hypergraphs
whose connected components have a single vertex, to $1$.
Restricting the above map \eqref{eq:newq-cop} to $H^\circ$ we get a coproduct
$H^\circ \overset{\delta^{\prime \prime}}{\pil} H^\circ \te H^\circ$.

  \begin{proposition}[L.~Foissy \cite{Fo-Hyp}, but slightly more general]
    \label{pro:newq-ex}
$(H^\circ, \mu, \delta^{\prime \prime}, \eta,
\vareps^{\prime \prime})$ is a bialgebra, and the bialgebra $(H, \mu, \Delta^{\prime \prime}, \eta, \eps)$ is a comodule bialgebra over the former.

In consequence $(H^\circ, \mu, \Delta^{\prime \prime},
\delta^{\prime \prime})$ is a double bialgebra.
\end{proposition}

\begin{remark}
The extraction-contraction bialgebra in \cite[Section 1.3]{Fo-Hyp} 
is (essentially) a quotient of this bialgebra $(H^\circ, \mu, \delta^{\prime \prime})$. First quotient it by:
\begin{itemize}
 \item The differences between a hypergraph and its reduced version by
   removing: \begin{itemize}
             \item multiple edges (but keeping one edge),
             \item any edge containing a single vertex.
   \end{itemize} \end{itemize}
 Then modify every hypergraph in the resulting bialgebra
 by:  \begin{itemize}
      \item adding an empty edge, and
      \item   replace every one vertex  isolated component $(\emptyset, v)$ with
  $(\{v\},v)$.
\end{itemize}
\end{remark}

  \begin{proof}[Proof of Proposition \ref{pro:newq-ex}.]
    1. We first see that $\delta^{\prime \prime}$ is coassociative.
  Applying $\id \te \delta^{\prime \prime}$ to \eqref{eq:newq-cop}
we get 
\begin{equation} \label{eq:newq-ende}  (\prod_{w \in W} (E_w,V_w)) \te (\prod_{x \in X} (E_x,W_x)) \te (E,X).\end{equation}
where we sum over surjections $V \ppil W \ppil X$ with fiber hypergraphs
connected.
Applying $\delta^{\prime \prime} \te \id$ to 
\[ (\prod_{x \in X} (E_x,V_x)) \te (E,X),\] 
we get
\begin{equation} \label{eq:newq-deen}
  (\prod_{w \in W} (E_{x,w},V_{x,w}))
  \te (\prod_{x \in X} (E_x,W_x)) \te (E,X).
  \end{equation}
Where we sum over surjections $V_x \ppil W_x$ with fiber hypergraphs 
connected. We then see that \eqref{eq:newq-ende} and
\eqref{eq:newq-deen} become equal. (We do not really need
here that the fiber hypergraphs are connected. This is used
in the next part.)

\medskip
2. To show the cointeraction, note first that Property 1) of
Section \ref{sec:coint} holds just like in the proof of Theorem
\ref{thm:EC-cointer1}.
We then need to show Property 2). Let us describe the composite
$(\id \te \Delta^{\prime \prime}) \circ \delta^{\prime \prime}$. By applying $(\id \te \Delta^{\prime \prime})$ to
\eqref{eq:newq-cop} we get
\begin{equation} \label{eq:newq-1Dd}
  \bigoplus_{\begin{matrix} V \ppil W \\ W = W_1 \sqcup W_2 \end{matrix}}
  (\prod_{w \in W} (E_w,V_w)) \te (E,W_1) \te (E,W_2).
\end{equation}

The coproduct $\Delta^{\prime \prime}$ sends $(E,V)$ to:
\[ \underset{V = V_1 \sqcup V_2}\bigoplus (E,V_1) \te (E,V_2).\]
The composition $(\delta^{\prime \prime} \te \delta^{\prime \prime}) \circ
\Delta^{\prime \prime} $ sends this further
to
\[ \bigoplus_{\begin{matrix} V_1 \ppil W_1\\ V_2 \ppil W_2 \end{matrix}}
  (\prod_{w \in W_1} (E_w,V_w)) \te (E,W_1) \te
  (\prod_{w \in W_2} (E_w,V_w)) \te (E,W_2). \]
Multiplying the first and third term, this corrsponds to the terms
in \eqref{eq:newq-1Dd}.
\end{proof}

We thus get a double bialgebra morphism
\[ \chi^{\prime \prime} : (H^\circ, \mu, \Delta^{\prime \prime},
  \delta^{\prime \prime}) \pil  (\QQ[x], \cdot, \Delta, \delta). \]
Foissy \cite{Fo-Hyp} shows that the image of a hypergraph $h$
is the classical chromatic polynomial for graphs, if each edge is replaced
by the complete simple graph on its set of vertices. So
$\chi_h^{\prime \prime}(n)$ counts the number of colorings with $n$ colors
where each edge is {\it rainbow}: all its vertices have different
colors.
In particular any hypergraph $h$ with exactly one vertex (and no empty edges)
has polynomial $\chi_h^{\prime \prime}(x) = x$, irrespective of how
many edges it has.

\begin{example} 
  The hypergraphs $\trekantll$ and $\treell$ have the same
  $\chi^{\prime \prime}$-chromatic polynomials.
  However all four $\chi^{\prime \prime}$-chromatic polynomials refines the view.
  The first gives the quartet:
  \[ h : \trekantll, \quad h^d : \treel, \quad h^c : \veel \pointl,
    \quad h^{cd} : \edgell \pointl,  \]
  with associated polynomials:
  \[ x(x-1)(x-2)^2, \quad x(x-1)(x-2), \quad x^2(x-1)^2, \quad x^2(x-1) . \]
  The second gives the quartet:
  \[ g : \treell, \quad h^d : \bbox, \quad h^c : \treeltl, \quad h^{cd} : \bboxt. \]
  with associated polynomials:
  \[ x(x-1)(x-2)^2, \quad x(x-1)(x-2)(x^2-5x + 7), \quad x(x-1)(x-2)^2,
    \quad x(x-1)(x-2)(x^2-5x+7) . \]
\end{example}


\subsection{Question on cointeraction}

The bialgebra $(H^\circ,\mu,\Delta^\prime)$
does not seem to come with a
cointeraction.
To see the problem, consider the single edge. The natural coproduct
on a cointeracting algebra would be:
\[ \linvert \, \overset{\delta^\prime}\longmapsto \, \tobull \te \linvert +
  \linvert \te \bullet. \]
Considering the half-edge we should also expect to
sum over the subsets of the
edges, giving the two term coproduct:
\begin{equation} \label{eq:pro-half}
  \pointl \, \overset{\delta^\prime}\longmapsto \, \bullet \te \pointl + \pointl
  \te \bullet.
  \end{equation}
  
  The requirement for cointeraction is the equation
\begin{equation*} 
  (\ben_B \te \Delta^\prime) \circ \delta^\prime
  = m_{13,2,4} \circ (\delta^\prime \te \delta^\prime)
\circ \Delta^\prime.
  \end{equation*}
  However applying this to the edge in $(\Hy^\circ, \mu, \Delta^\prime)$, the left
  side gets $6$ terms, and the right side $8$ terms.
  The map $\delta^\prime \te \delta^\prime $ applied to \eqref{eq:pro-cop3} gives
  two extra terms due to the half-edge \eqref{eq:pro-half},
  compared to $\delta \te \delta$ in $(\Hy^\circ, \mu, \Delta)$.

\begin{question}
  Is it possible in some way to get cointeraction for the bialgebra
  $(\Hy^\circ, \mu, \Delta^\prime)$? If not, why is it not possible?
\end{question}

\begin{remark} \cite[Prop. 2.7]{Fo-Hyp} shows, under some extra conditions,
  that there cannot be a cointeraction.
\end{remark}

All the bialgebras we give are commutative. Some are co-commutative and some
not.
\begin{question} 
  Are there bialgebras of hypergraphs which are not commutative and not
  co-commutative?
\end{question}

\bibliographystyle{amsplain}
\bibliography{biblio}
\end{document}